\documentclass{amsart}
\usepackage{graphicx}

\usepackage{hyperref}

\newtheorem{theorem}{Theorem}[section]
\newtheorem{lem}[theorem]{Lemma}
\newtheorem{corollary}[theorem]{Corollary}
\newtheorem{proposition}[theorem]{Proposition}

\theoremstyle{definition}

\newcommand{\F}{\mathbb{F}}

\def\Tr{\mathop{\rm Tr}\nolimits}

\theoremstyle{remark}
\newtheorem{remark}[theorem]{Remark}

\numberwithin{equation}{section}

\begin{document}

\title{The punctured dodecacode is unique}%
\thanks{In memory of Stefan Dodunekov, on the occasion of his 80th birthday}
\author{Markus Grassl}
\address{International Centre for Theory of Quantum Technologies, University of Gdansk, Gdansk, Poland}
\email{markus.grassl@ug.edu.pl}
\author{Denis Krotov}
\address{Sobolev Institute of Mathematics, Novosibirsk 630090, Russia}
\email{dk@ieee.org}
\author{Lin Sok}
\address{School of Physical and Mathematical Sciences, Nanyang Technological University,
21 Nanyang Link, Singapore 637371, Singapore}
\email{lin.sok@ntu.edu.sg}
\author{Patrick Sol\'e}
\address{I2M (Aix Marseille Univ, CNRS), Marseilles, France}
\email{sole@enst.fr}
\subjclass[2000]{94B25, 05E30}
 \keywords{dodecacode, additive code, trace Hermitian duality, uniformly packed code, completely regular code, Doob graph, strongly regular graph}
\begin{abstract}
The punctured dodecacode is an additive $4$-ary code of length $11$ and distance $5$ which is uniformly packed. We show that a code with the same weight distribution is equivalent to it. This code is also shown to be nonlinear.

We also establish the nonexistence of analogues of  the dodecacode and the punctured dodecacode  in Doob graphs. To that end, we classify two-weight codes of weights $6$ and $8$ in Doob and $4$-ary Hamming graphs of diameter $9$ and the corresponding strongly regular graphs.
\end{abstract}
\maketitle
\section{Introduction}
The dodecacode, introduced for the needs of quantum computing in \cite{CRSS:quantum}, is the most famous additive $\F_4$ code. It is trace Hermitian self-dual. It has length $12$ and distance~$6$. The parameters of possible uniformly packed quaternary codes were found in \cite{BZZ:1974:UPC} to belong to an infinite family of codes of length $\frac{2^{2m+1}+1}{3}$, and codimension $2m+1$. The punctured dodecacode is the case $m=2$ of this family. As a uniformly packed code, it is completely regular \cite{CRCinDRG}, and its coset graph is a distance-regular graph of diameter~$3$ \cite{SKS:drg,graph1024}. In this note, we show that it is unique (not only in the Hamming scheme $H(11,4)$, but also in all schemes $D(m,11-2m)$ with the same algebraic parameters), and not $\F_4$-linear.

The main part of the note consists of two sections. In Section~\ref{s:H},
we show that the punctured dodecacode is unique as an additive $4$-ary length-$11$ code (i.e., a code in $H(11,4)=D(0,11)$)
with given weight distribution. This part includes also some partial classification results for additive $4$-ary codes (Section~\ref{s:part}).
In Section~\ref{s:D}, we establish the nonexistence of such codes in
the Doob graphs $D(m,11-m)$, $0<m\le 5$, and of analogs of the dodecacode in $D(m,12-m)$, $0<m\le 6$.
This part includes also a classification of additive two-weight codes in $D(m,9-m)$, $0\le m\le 4$ (Section~\ref{s:srg}) and the corresponding strongly regular graphs.
The appendix contains the results of the classification up to equivalence of the following codes, which were used for obtaining the results in Section~\ref{s:D}: weight-$\{6,8\}$ codes in Doob graphs $D(m,9-2m)$; additive $4$-ary codes of length~$6$, distance at least~$3$ and size at least~$2^6$.

\section{Codes in Hamming metric space}\label{s:H}
\subsection{Preliminaries}
An additive $\F_4$-code of length $n$ is an additive subgroup of~$\F_4^n$. The Hamming weight is denoted by $wt(.)$. The trace denoted by $\Tr()$, is defined by $\Tr(z)=z+z^2$.
The two inner products we consider here are the Hermitian inner product $$\langle x,y\rangle =\sum_i x_iy_i^2,$$ and the trace Hermitian inner product $$\Tr(\langle x,y\rangle )=\sum_i x_iy_i^2+x_1^2y_i.$$
Thus $\Tr(\langle x,y\rangle )=\langle x,y\rangle +\langle y,x\rangle$. The duals of a code $C$ wrt these two inner products are denoted by $C^\perp$ and $C^{\perp_T}$, respectively.
The dodecacode is an additive $\F_4$ code of length $12$ and binary
dimension $12$ which is self-dual with respect to the trace Hermitian inner product. It is cyclic and puncturing at any coordinate
yields a code $C$ of parameters $(11,2^{12},5)_4$. Its dual weight distribution in Magma notation is
 $$\texttt{[ <0,1>, <6,198>, <8,495>, <10,330> ]}.$$

\subsection{Results}
The proof of the main result is based on trace Hermitian duality. We need a pair of Lemmas.
{\lem \label{l:2.1} If an additive $\F_4$-code $C$ is even, then $C$ is self-orthogonal for the trace Hermitian inner product}

\begin{proof}(see also \cite[Th. 4]{CRSS:quantum})
%

  The semilinearity of the Hermitian inner product implies
  $$
    \langle x+y,x+y\rangle=\langle x,x\rangle +\langle y,y\rangle
    +\langle x,y\rangle +\langle y,x\rangle.
  $$
  Using $\langle x,x\rangle \equiv wt(x) \pmod{2}$ and $\langle x,y\rangle +\langle y,x\rangle=\Tr(\langle x,y\rangle )$
  we obtain (see also eq. (7) in \cite{CRSS:quantum})
  $$
    wt(x+y)\equiv wt(x)+wt(y)+ \Tr(\langle x,y\rangle)  \pmod{2}.
  $$
  If $C$ is even, this implies $\Tr(\langle x,y\rangle)=0$.
\end{proof}

The following is proved in~\cite{DanPar:2006}.
{\lem \label{l:2.2} There is a unique self-dual code of length $11$ and distance~$5$.}

Note that equivalence includes all monomial transformations as well as
conjugation of coordinates.

We are now ready for a characterization of the punctured dodecacode by
its weight distribution. A characterization by the minimum distance is
hopeless in view of the existence of the quadratic residue code of
parameters $[11,6,5]_4$ (but see the next subsection for partial
classification results).

{\theorem The punctured dodecacode is the unique additive code with the dual weight distribution
  $$
  \normalfont\texttt{[ <0,1>, <6,198>, <8,495>, <10,330> ]}.
  $$
}

\begin{proof}
Consider an additive code $C$ of length $11$ with the same weight
distribution as the punctured dodecacode. Its dual, say~$D$, contains
$2^{10}$ codewords. It is even since its weights are $6$,
$8$,~$10$. By Lemma~\ref{l:2.1}, $D$ is trace Hermitian
self-orthogonal. Consider a weight-$5$ codeword $x$ in
$C=D^{\perp_T}$. Such a word exists by the MacWilliams identity. The code $N=D \cup (x+D)$ is additive and self orthogonal of size $2^{11}$, hence self-dual. (Note that $\Tr(\langle x,x\rangle )=0$). It has distance~$5$. By Lemma~\ref{l:2.2}, it is unique, up to equivalence. By construction, $D$ is the even part of~$N$, as an even subcode of size half the size of~$N$. The result follows.
\end{proof}

We could prove the following result by inspection, but a conceptual proof is better. The following was shown to the last author by Jon-Lark Kim.

{\proposition The punctured dodecacode is not $\F_4$-linear. } 
\begin{proof}
The code $D$ of the preceding proof, being even and $\F_4$-linear would be Hermitian self-orthogonal by \cite{MWOSW:78}. But there is no Hermitian self-orthogonal length~$11$ distance~$6$ code over~$\F_4$
by~\cite{BouOst:2005}.
\end{proof}

\subsection{Partial classification results}\label{s:part}
In order to show the uniqueness of the punctured dodecacode, we have also
obtained some partial classification results.

Shortening an additive code over $\F_4$ with parameters
$C=(n,2^{k_2},d)_4$ reduces the dimension at most by two, i.\,e.,
results in a code $C'=(n-1,2^{k_2'},\ge d)_4$ where
$k_2'\in\{k_2-2,k_2-1,k_2\}$. 
In order to find all inequivalent additive codes
$(11,2^{12},5)_4$ over $\F_4$, we start with the code $(5,1,5)_4$
generated by the all-one vector. Then we try to increase the dimension
by one, while the minimum distance is at least~$5$. Having found all
codes $(n,2^k,d\ge 5)_4$ with covering radius $\le 4$, i.\,e., codes
for which the dimension cannot be increased, we add a zero coordinate
to the codes and continue.  We exclude codes whose dimension is too
small to reach $(11,2^{12},5)_4$.

Table \ref{t:d5} shows the number of inequivalent additive codes
$(n,2^k,5)_4$ that we have found.
We only found two codes $(11,2^{12},5)_4$, corresponding to the
punctured dodecacode and the $\F_4$-linear quadratic residue code.
\begin{table}
 \begin{alignat*}{5}
%
\begin{array}{c|*{13}{c}}
  n\backslash k &0 &1 & 2 & 3 & 4 & 5 & 6 &7 &8 &9 & 10 & 11 & 12\\\hline
   5  & 1 & 1 & 1  & 0 \\
   6  & 1 & 2 & 5  & 1  & 0  & 0       \\
   7  & 1 & 3 & 15 & 32 & 43 & 1        & 0       \\
   8  & 1 & 4 & 34 & 322&5114&   26299 & 1579 & 0         & 0    & \\ 
   9  & 1 & 5 & 65 & 1698   & 139546    &          &      & \!\!\!\ge 89116 & 2298 &         0 &  0 \\
  10  & 1 & 6 & 110& 6139   &    &          &      &           &      & \!\!\!\ge 20935 & 37 & 0      & 0\\
  11  & 1 & 7 & 171&    &    &          &      &           &      &           &    & \!\!\!\ge 24 & 2
\end{array}
\end{alignat*}
\caption{The number $N(n,k,5)$ of equivalence classes of additive $(n,2^k,{\ge} 5)_4$ codes. According to our approach, we classify $(n,2^k,{\ge} 5)_4$ codes by lengthening $N(n{-}1,k{-}2,5)+N(n{-}1,k{-}1,5)+N(n{-}1,k,5)$ codes of length~$n-1$.
The notation $\ge N_0$ means that we lengthened only $N(n{-}1,k{-}1,5)+N(n{-}1,k,5)$ codes of size~$2^{k-1}$ and $2^{k}$ and $N_0$ is the number of codes having at least one coordinate with at most two symbols.}\label{t:d5}
\end{table}

There is also a unique cyclic code $(7,2^5,5)_4$ with generator matrix
\begin{alignat*}{5}\def\w{\omega}
  \begin{pmatrix}
    1 &    0 &    \w &  \w^2 &    0 & \w^2 &    1\\
   \w &    0 &     0 &     1 & \w^2 &   \w &    1\\
    0 &    1 &    \w &     1 & \w^2 & \w^2 &   \w\\
    0 &   \w &    \w &  \w^2 &    1 &    0 & \w^2\\
    0 &    0 &     1 &  \w^2 &   \w &    1 &   \w
  \end{pmatrix},
\end{alignat*}
where $\omega^2=\omega+1$. Its weight distribution is
$$
\texttt{[ <0,1>, <5,21>, <6,7>, <7,3> ]}.
$$
Note that the best $\F_4$-linear codes
have parameters $[7,2,5]_4$ and $[7,3,4]_4$. Hence, this code has twice
as many codewords as the largest $\F_4$~linear code of length $7$~and
distance~$5$.

\newcommand\ZZ{\mathbb{Z}}
\section{Codes in Doob metric space}\label{s:D}

\subsection{Preliminaries}
A set $C$ of vertices in a connected regular graph is called a completely regular code with covering radius~$\rho$ and intersection array
$I=\{b_0,$ $\ldots,$ $b_{\rho-1};$ $c_1,$ $\ldots,$ $c_{\rho}\}$
(for short, CR-$\rho$, $I$-CR, or just CR code)
if each vertex in $C^{(i)}$ has exactly
$b_i$ neighbors in $C^{(i+1)}$ and exactly
$c_i$ neighbors in $C^{(i-1)}$, $b_\rho=c_0=0$,
where $C^{(i)}$ is the set of vertices at distance~$i$ from~$C$ and $\rho$ is the largest~$i$ such that $C^{(i)}$ is nonempty.
A connected graph is called \emph{distance-regular}
with intersection array~$I$ if every singleton is an $I$-CR code. Distance-regular graphs of diameter~$2$ are called \emph{strongly regular}.

The Doob graph $D(m,n)$ is the Cartesian product of $m$ copies of the Shrikhande graph on $16$ vertices
and $n$ copies of the complete graph $K_4$ of order~$4$. Usually, for Doob graphs it is implied
that $m>0$, while $D(0,n)$ is the $4$-ary Hamming graph $H(n,4)$.
It is known that $D(m,n)$ is a distance-regular graph with intersection array
$\{3d,3d{-}3,\ldots,3;1,2,\ldots,d\}$, where $d=2m+n$ is its diameter; in particular, all $D(m,n)$ have the same intersection array as $H(2m+n,4)$.
If we associate the vertices of $D(m,n)$
with the elements of the module
$ M = \mathbb{R}_{16}^m \times \F_4^n$ over the Galois ring $\mathbb{R}_{16}=\mathrm{GR}(4^2)$,
then we can define linear codes as submodules of~$ M $
and additive codes as subgroups of the additive group of~$ M $.
In particular, treating $\mathbb{R}_{16}$ as a module over $\ZZ_4$, we see that such additive codes are submodules of $(\ZZ_4^2)^{m} \times (\ZZ_2^2)^{n}$.
In \cite{Kro:perfect-doob}, the concept of additive codes in $D(m,n)$ was further generalized by allowing some of the last $n$ coordinates to have the structure $\ZZ_4$ instead of~$\ZZ_2^2$. The formal definition below is based on that extended concept of additive codes.

We first recall that the \emph{Caley graph} on an Abelian group $G$ with a connecting set $S\subset G\setminus\{\mathrm{Id}\}$, $S=S^{-1}$, is the graph with the vertex set~$G$
where two vertices $a,b\in G$ are adjacent if and only if $b-a\in S$.
Let the \emph{Doob graph}
$D(m,n'+n'')$ be defined as the Cayley graph
on the additive group of the module
$V=(\ZZ_4^2)^m \times (\ZZ_2^2)^{n'} \times \ZZ_4^{n''}$
with connecting set $S=S^*\cup S'\cup S''$,
where $S^*$ ($S'$, $S''$) consists of vectors of~$V$
with one of the first $m$ symbols (the next $n'$ symbols, the last $n''$ symbols, respectively)
lying in $\{01,10,11,03,30,33\}$ ($\{01,10,11\}$, $\{1,2,3\}$, respectively) and all other symbols are zero, $00$ or $0$
(we naturally consider vectors of~$V$
as words of length $m+n'+n''$ where the first $m$ symbols are from $\ZZ_4^2$,
the next $n'$ symbols are from $\ZZ_2^2$, and the last $n''$ symbols are from $\ZZ_4$).
The (Doob) weight $\mathrm{wt}(x)$ of a vector~$x$ in~$V$
is the sum of the weights of all symbols of~$x$, where
$$
\mbox{for $x^*\in \ZZ_4^2$, }
\mathrm{wt}(x^*) = \left\{\begin{array}{rl}
0& \mbox{ if $x^*=00$,} \\
1& \mbox{ if $x^*\in\{01,03,10,30,11,33\}$,} \\
2& \mbox{ otherwise;}
\end{array}\right.
$$
$$
\mbox{for  $x'\in \ZZ_2^2$, }
\mathrm{wt}(x') = \left\{\begin{array}{rl}
0& \mbox{ if $x'=00$,} \\
1& \mbox{ if $x'\in\{01,10,11\}$;}
\end{array}\right.
$$
$$
\mbox{for  $x''\in \ZZ_4$, }
\mathrm{wt}(x'') = \left\{\begin{array}{rl}
0& \mbox{ if $x''=0$,} \\
1& \mbox{ if $x''\in\{1,2,3\}$.}
\end{array}\right.
$$
The (Doob) distance between $x,y\in V$ is defined as $\mathrm{wt}(y-x)$
and coincides with the natural shortest-path distance in the graph $D(m,n'+n'')$.

As in the case of codes in the Hamming space $H(n,4)$, the weight
distributions of a code and its dual are connected by the MacWilliams
transform (the formulas are the same as in $H(2m+n'+n'',4)$). 
However, the duality should be defined properly, with respect to a special inner product~$\langle \cdot,\cdot\rangle_{_-}$ (see \cite[Theorem~6.6]{BesKro:Ch6}):
for
$x^*,y^*\in (\ZZ_4^2)^m$,
$x',y'\in \ZZ_4^{n'}$,
$x'',y''\in (\ZZ_2^2)^{n''}$
\begin{equation}\label{eq:ip}
 \big\langle(x^*,x',x''),(y^*,y',y'')\big\rangle_-
=\sum_{i=1}^{m}
(x^*_{2i-1}y^*_{2i-1}-x^*_{2i}y^*_{2i})
+2\cdot\!\!\sum_{i=1}^{2n'}x'_{i}y'_{i}
+\!\!\sum_{i=1}^{n''}x''_{i}y''_{i}.
\end{equation}
In particular, similarly to the codes in $H(n,q)$, a code with distance at least~$3$ is CR-$2$ (completely regular with covering radius~$2$) if and only if its dual has two nonzero weights.


\subsection{Two-weight codes of size 64}\label{s:srg}
Our first aim is to classify additive codes of size~$64$ in $D(m,n'+n'')$, $2m+n'+n''=9$, with non-zero weights $6$ and~$8$ and dual distance at least~$3$.
The last condition means that either $m=n''=0$ or the code has at least one element $z$
of order~$4$, which can only be possible if $2m+n''\ge 6$ (otherwise, $0<\mathrm{wt}(2z)<6$).
In Appendix~\ref{s:all}, we list all, up to equivalence, additive weight-$\{6,8\}$ (two-weight with weights $6$ and~$8$ and one-weight with weight $6$ or~$8$) codes in Doob graphs $D(m,n'+n'')$ of diameter $2m+(n'+n'')=9$ such that $2m+n''\ge 6$ and in $D(0,9+0)=H(9,4)$.
The approach of the classification is rather straightforward.
We start with the collection consisting of one trivial code having
only the all-zero codeword. For each code $C$ from the collection, we
try to continue it by adding a new element $c$ to the generator set in
all possible ways and checking if the resulting code
$\mathrm{span}(C\cup\{c\})$ has the required weights and is not equivalent to any code from our collection.
If it is a new weight-$\{6,8\}$ code, we add it to the collection.
When all codes from the collection are checked for continuation, the process stops.

In Tables~\ref{t:410}, \ref{t:303}, \ref{t:090}, as a summary, we list only weight-$\{6,8\}$ codes of size~$64$. The following proposition summarizes the computational results.

\begin{proposition}
 In $D(m,n'+n'')$, $2m+n'+n''=9$,
there are exactly $26$ equivalence classes of weight-$\{6,8\}$ codes of size~$64$,
$B_1$, \ldots, $B_6$ in $D(4,1+0)$ (Table~\ref{t:410}),
$C_1$, \ldots, $C_8$ in $D(3,0+3)$ (Table~\ref{t:303}),
$D_1$, \ldots, $D_8$ in $D(0,9+0)=H(9,4)$ (Table~\ref{t:090}).
The weight distribution of all these codes is $\{0^{1}, 6^{36},
8^{27}\}$; the dual codes (with respect to the inner product~\eqref{eq:ip}) are completely regular with intersection array $\{27,16;1,12\}$; the coset graphs of the dual codes are strongly regular with parameters
$(64, 27, 10, 12)$;
these strongly regular graphs belong to only $7$ equivalence classes, corresponding to the groups
$\{B_1,C_2,D_2,D_6\}$,
$\{B_2,B_3,C_1,C_3,C_8,D_1,D_5\}$,
$\{B_4\}$,
$\{B_5,C_4,C_6\}$,
$\{B_6,C_5,C_7\}$,
$\{D_3,D_4,D_7\}$,
$\{D_8\}$ of codes.
\end{proposition}

\begin{table}
$$ B_1: \left( \begin{array}{cccc|c|}
21&10&10&10&10\\
31&01&01&01&01\\ \hline
20&02&20&00&00\\
02&22&02&00&00\\
\end{array} \right),
\qquad
  B_2: \left( \begin{array}{cccc|c|}
21&10&10&10&10\\
12&11&01&01&01\\ \hline
20&02&20&00&00\\
22&20&02&00&00\\
\end{array} \right),$$
$$ B_3: \left( \begin{array}{cccc|c|}
21&21&10&00&10\\
20&10&01&10&01\\
12&01&01&01&01\\ \hline
\end{array} \right),
\qquad
 B_4: \left( \begin{array}{cccc|c|}
21&21&10&00&10\\
20&10&01&10&01\\
13&01&01&01&11\\ \hline
\end{array} \right),$$
$$ B_5: \left( \begin{array}{cccc|c|}
21&21&10&00&10\\
20&10&01&10&01\\
32&03&01&01&10\\ \hline
\end{array} \right),
\qquad
  B_6: \left( \begin{array}{cccc|c|}
21&21&10&00&10\\
31&20&01&10&00\\
21&33&01&01&01\\ \hline
\end{array} \right).$$
 \caption{Generator matrices of additive weight-$\{6,8\}$ codes of
   type $\ZZ_4^2\ZZ_2^2$ ($B_1$, $B_2$) and type $\ZZ_4^3$ ($B_3$,
   $B_4$, $B_5$, $B_6$) in $D(4,1+0)$. Here and further, the two vertical lines in the matrix separate the three groups of coordinates (one or two of which can be empty), while the horizontal line separates the rows of order $4$ and rows of order-$2$}\label{t:410}
\end{table}

\begin{table}
$$ C_1: \left( \begin{array}{ccc||ccc}
21&10&10&1&1&0\\
12&01&01&1&0&1\\ \hline
22&02&20&0&0&0\\
02&22&02&0&0&0\\
\end{array} \right),
\qquad
 C_2: \left( \begin{array}{ccc||ccc}
21&10&10&1&1&0\\
10&21&01&1&0&1\\ \hline
22&02&20&0&0&0\\
02&22&02&0&0&0\\
\end{array} \right),$$
$$ C_3: \left( \begin{array}{ccc||ccc}
21&21&21&0&0&0\\
31&31&00&1&1&0\\
02&10&10&1&0&1\\ \hline
\end{array} \right),
\qquad
 C_4: \left( \begin{array}{ccc||ccc}
21&21&21&0&0&0\\
31&31&00&1&1&0\\
22&30&10&1&0&1\\ \hline
\end{array} \right)$$
$$ C_5: \left( \begin{array}{ccc||ccc}
21&21&21&0&0&0\\
30&20&10&1&1&0\\
33&11&20&1&0&1\\ \hline
\end{array} \right),
\qquad
  C_6: \left( \begin{array}{ccc||ccc}
21&21&10&2&0&0\\
20&10&01&1&1&0\\
12&01&01&1&0&1\\ \hline
\end{array} \right),$$
$$ C_7: \left( \begin{array}{ccc||ccc}
21&21&10&2&0&0\\
20&10&01&1&1&0\\
33&20&11&1&0&1\\ \hline
\end{array} \right),
\qquad
 C_8: \left( \begin{array}{ccc||ccc}
21&21&10&2&0&0\\
21&30&01&1&1&0\\
30&02&01&1&0&1\\ \hline
\end{array} \right).$$
 \caption{Generator matrices of additive weight-$\{6,8\}$ codes of type $\ZZ_4^2\ZZ_2^2$ ($C_1$, $C_2$) and type $\ZZ_4^3$ ($C_3$, $C_4$, $C_5$, $C_6$, $C_7$, $C_8$) in $D(3,0+3)$}\label{t:303}
\end{table}

\begin{table}
\small
$$ D_1: \left( \begin{array}{@{\ }c@{\,}c@{\,}@{\,}c@{\,}c@{\,}c@{\,}c@{\,}c@{\,}c@{\,}c@{\ }}
10&10&10&10&10&10&00&00&00\\
01&01&01&01&01&01&00&00&00\\
01&01&10&10&00&00&10&10&00\\
11&11&01&01&00&00&01&01&00\\
11&01&01&00&11&00&10&00&10\\
10&11&11&00&10&00&01&00&01\\
 \end{array} \right),
\quad
 D_2: \left( \begin{array}{@{\ }c@{\,}c@{\,}@{\,}c@{\,}c@{\,}c@{\,}c@{\,}c@{\,}c@{\,}c@{\ }}
10&10&10&10&10&10&00&00&00\\
01&01&01&01&01&01&00&00&00\\
01&01&10&10&00&00&10&10&00\\
11&11&01&01&00&00&01&01&00\\
11&01&01&00&11&00&10&00&10\\
01&00&01&10&10&00&01&00&01\\
 \end{array} \right),$$
$$ D_3: \left( \begin{array}{@{\ }c@{\,}c@{\,}@{\,}c@{\,}c@{\,}c@{\,}c@{\,}c@{\,}c@{\,}c@{\ }}
10&10&10&10&10&10&00&00&00\\
01&01&01&01&01&01&00&00&00\\
01&01&10&10&00&00&10&10&00\\
01&10&01&00&10&00&01&01&00\\
01&11&11&00&01&00&10&00&10\\
10&10&00&01&01&00&01&00&01\\
 \end{array} \right),
\quad
 D_4: \left( \begin{array}{@{\ }c@{\,}c@{\,}@{\,}c@{\,}c@{\,}c@{\,}c@{\,}c@{\,}c@{\,}c@{\ }}
10&10&10&10&10&10&00&00&00\\
01&01&01&01&01&01&00&00&00\\
01&01&10&10&00&00&10&10&00\\
01&10&01&00&10&00&01&01&00\\
01&11&11&00&01&00&10&00&10\\
11&00&01&11&01&00&01&00&01\\
 \end{array} \right),$$
$$ D_5: \left( \begin{array}{@{\ }c@{\,}c@{\,}@{\,}c@{\,}c@{\,}c@{\,}c@{\,}c@{\,}c@{\,}c@{\ }}
10&10&10&10&10&10&00&00&00\\
01&01&01&01&01&01&00&00&00\\
01&01&10&10&00&00&10&10&00\\
11&01&01&00&11&00&01&01&00\\
11&11&01&01&00&00&10&00&10\\
10&11&11&00&10&00&01&00&01\\
 \end{array} \right),
\quad
 D_6: \left( \begin{array}{@{\ }c@{\,}c@{\,}@{\,}c@{\,}c@{\,}c@{\,}c@{\,}c@{\,}c@{\,}c@{\ }}
10&10&10&10&10&10&00&00&00\\
01&01&01&01&01&01&00&00&00\\
01&01&10&10&00&00&10&10&00\\
11&01&01&00&11&00&01&01&00\\
11&11&01&01&00&00&10&00&10\\
01&00&01&10&10&00&01&00&01\\
 \end{array} \right),$$
$$ D_7: \left( \begin{array}{@{\ }c@{\,}c@{\,}@{\,}c@{\,}c@{\,}c@{\,}c@{\,}c@{\,}c@{\,}c@{\ }}
10&10&10&10&10&10&00&00&00\\
01&01&01&01&10&00&10&00&00\\
11&11&01&10&01&00&00&10&00\\
11&01&01&00&00&01&01&01&00\\
10&01&11&01&00&01&00&00&10\\
11&00&01&11&11&00&01&00&01\\
 \end{array} \right),
\quad
 D_8: \left( \begin{array}{@{\ }c@{\,}c@{\,}@{\,}c@{\,}c@{\,}c@{\,}c@{\,}c@{\,}c@{\,}c@{\ }}
10&10&10&10&10&10&00&00&00\\
01&01&01&01&10&00&10&00&00\\
11&11&01&10&01&00&00&10&00\\
11&01&01&00&00&01&01&01&00\\
10&11&00&11&01&01&00&00&10\\
00&11&11&01&01&00&01&00&01\\
 \end{array} \right).$$
 \caption{Generator matrices of additive weight-$\{6,8\}$ codes  in $D(0,9+0)=H(9,4)$}\label{t:090}
\end{table}

\begin{remark}
 By construction, the graphs from codes of group type $\ZZ_4^i
 \ZZ_2^j$ are strongly regular Cayley graphs over $\ZZ_4^i \ZZ_2^j$
 (in our case, $(i,j)$ is $(0,6)$, $(2,2)$, or $(3,0)$); however, we
 cannot say that we have found all such graphs. The reason is that
 being representable as a coset graph of a code in a Doob graph or a Hamming graph $H(n,4)$ implies certain relations between vectors from the connecting set of a Cayley graph. More strongly regular Cayley graphs with the same parameters can hypothetically
be constructed as coset graphs of binary linear of $\ZZ_2\ZZ_4$-linear two-weight codes.
\end{remark}

\subsection{Lengthening to codes in Doob graphs of diameter 11 and 12}
\label{s:D11}
We aim to classify all additive codes of size $1024$ with nonzero weights $6$, $8$, $10$ in Doob graphs $D(m,n'+n'')$, $m>0$, of diameter~$2m+n'+n''=11$.
If such a code $C_{11}$ exists, then shortening it in one of the first $m$ positions or two of the last $n'+n''$ positions leads to a weight-$\{6,8\}$ code $C_{9}$ of size~$64$ in $D(m-1,n'+n'')$,
 $D(m,(n'-2)+n'')$ (if $n'\ge2$),  $D(m,(n'-1)+(n''-1))$ (if $n',n''\ge1$), or $D(m,n'+(n''-2))$ (if $n''\ge2$). As such codes $C_{9}$ exist only in $D(4,1+0)$ and $D(3,0+3)$, we conclude that the only option for $C_{11}$ is $D(m,n'+n'')=D(5,1+0)$. Now, our approach is quite straightforward. Starting with one of the six weight-$\{6,8\}$ codes found in $D(4,1+0)$, we extend the code to $D(5,1+0)$ by adding $00$ to the beginning of each codeword. Next, we add a row to the generator matrix in all possible ways and check that the resulting code has a larger cardinality ($128$ or $256$ in the first iteration) and desired nonzero weights $6$, $8$, $10$. After filtering only inequivalent representatives from the found collection of codes, we repeat the same procedure and repeat iterating until no new codes are found. Summarizing the results, we have found only codes of size~$128$ and~$256$,
 with weight distributions
 $\{0^1,6^{42},8^{55},10^{30}\}$ and
 $\{0^1,6^{54},8^{111},10^{90}\}$, respectively; as a corollary, we have Proposition~\ref{p:doob11} below. The check matrices of inequivalent representatives are given in Tables~\ref{t:32} and~\ref{t:42}
 (note that we do not claim that there are no other additive weight-$\{6,8,10\}$ codes of size at least~$128$ in~$D(5,1)$; our search was restricted only to those codes that have a shortening of size~$64$ in~$D(4,1)$).

\begin{table}
$$ \left( \begin{array}{@{\,}c@{\,}c@{\,}c@{\,}c@{\,}c@{\,}|@{\,}c@{\,}|}
00&21&21&10&00&10\\
00&20&10&01&10&01\\
00&12&01&01&01&01\\
\hline 20&22&20&00&00&00\\
\end{array} \right),
\quad
 \left( \begin{array}{@{\,}c@{\,}c@{\,}c@{\,}c@{\,}c@{\,}|@{\,}c@{\,}|}
00&21&21&10&00&10\\
00&20&10&01&10&01\\
00&13&01&01&01&11\\
\hline 20&20&20&00&00&00\\
\end{array} \right),
\quad
\left( \begin{array}{@{\,}c@{\,}c@{\,}c@{\,}c@{\,}c@{\,}|@{\,}c@{\,}|}
00&21&21&10&00&10\\
00&20&10&01&10&01\\
00&13&01&01&01&11\\
\hline 20&22&00&02&00&00\\
\end{array} \right),$$
$$ \left( \begin{array}{@{\,}c@{\,}c@{\,}c@{\,}c@{\,}c@{\,}|@{\,}c@{\,}|}
00&21&21&10&00&10\\
00&20&10&01&10&01\\
00&13&01&01&01&11\\
\hline 20&02&20&02&00&00\\
\end{array} \right),
\quad
\left( \begin{array}{@{\,}c@{\,}c@{\,}c@{\,}c@{\,}c@{\,}|@{\,}c@{\,}|}
00&21&21&10&00&10\\
00&20&10&01&10&01\\
00&32&03&01&01&10\\
\hline 20&22&20&00&00&00\\
\end{array} \right),
\quad
 \left( \begin{array}{@{\,}c@{\,}c@{\,}c@{\,}c@{\,}c@{\,}|@{\,}c@{\,}|}
00&21&21&10&00&10\\
00&20&10&01&10&01\\
00&32&03&01&01&10\\
\hline 20&02&20&02&00&00\\
\end{array} \right),$$
$$ \left( \begin{array}{@{\,}c@{\,}c@{\,}c@{\,}c@{\,}c@{\,}|@{\,}c@{\,}|}
00&21&21&10&00&10\\
00&31&20&01&10&00\\
00&21&33&01&01&01\\
\hline 20&22&20&00&00&00\\
\end{array} \right),
\quad
 \left( \begin{array}{@{\,}c@{\,}c@{\,}c@{\,}c@{\,}c@{\,}|@{\,}c@{\,}|}
00&21&10&10&10&10\\
00&31&01&01&01&01\\
\hline 00&20&02&20&00&00\\
00&02&22&02&00&00\\
20&02&20&00&00&00\\
\end{array} \right),
\quad
\left( \begin{array}{@{\,}c@{\,}c@{\,}c@{\,}c@{\,}c@{\,}|@{\,}c@{\,}|}
00&21&10&10&10&10\\
00&12&11&01&01&01\\
\hline 00&20&02&20&00&00\\
00&22&20&02&00&00\\
20&02&20&00&00&00\\
\end{array} \right).$$
\caption{Lengthenings of type $\ZZ_4^3\ZZ_2^1$ and type $\ZZ_4^2\ZZ_2^3$}\label{t:32}
\end{table}
\begin{table}
$$
 \left( \begin{array}{@{\,}c@{\,}c@{\,}c@{\,}c@{\,}c@{\,}|@{\,}c@{\,}|}
00&21&10&10&10&10\\
00&31&01&01&01&01\\
21&11&11&10&00&01\\
\hline 00&20&02&20&00&00\\
00&02&22&02&00&00\\
\end{array} \right),
\quad
  \left( \begin{array}{@{\,}c@{\,}c@{\,}c@{\,}c@{\,}c@{\,}|@{\,}c@{\,}|}
00&21&10&10&10&10\\
00&12&11&01&01&01\\
21&11&11&10&00&01\\
\hline 00&20&02&20&00&00\\
00&22&20&02&00&00\\
\end{array} \right),
\quad
\left( \begin{array}{@{\,}c@{\,}c@{\,}c@{\,}c@{\,}c@{\,}|@{\,}c@{\,}|}
00&21&21&10&00&10\\
00&20&10&01&10&01\\
00&12&01&01&01&01\\
\hline 20&22&20&00&00&00\\
02&20&00&02&00&00\\
\end{array} \right)$$
$$
\left( \begin{array}{@{\,}c@{\,}c@{\,}c@{\,}c@{\,}c@{\,}|@{\,}c@{\,}|}
00&21&21&10&00&10\\
00&20&10&01&10&01\\
00&13&01&01&01&11\\
\hline 20&20&20&00&00&00\\
02&22&00&02&00&00\\
\end{array} \right),
\quad
 \left( \begin{array}{@{\,}c@{\,}c@{\,}c@{\,}c@{\,}c@{\,}|@{\,}c@{\,}|}
00&21&21&10&00&10\\
00&20&10&01&10&01\\
00&32&03&01&01&10\\
\hline 20&22&20&00&00&00\\
02&20&00&02&00&00\\
\end{array} \right),
\quad
  \left( \begin{array}{@{\,}c@{\,}c@{\,}c@{\,}c@{\,}c@{\,}|@{\,}c@{\,}|}
00&21&21&10&00&10\\
00&31&20&01&10&00\\
00&21&33&01&01&01\\
\hline 20&22&20&00&00&00\\
02&22&02&02&00&00\\
\end{array} \right),$$
%
$$ \left( \begin{array}{ccccc|c|}
00&21&10&10&10&10\\
00&31&01&01&01&01\\
\hline 00&20&02&20&00&00\\
00&02&22&02&00&00\\
20&02&20&00&00&00\\
02&22&02&00&00&00\\
\end{array} \right),
\qquad
 \left( \begin{array}{ccccc|c|}
00&21&10&10&10&10\\
00&12&11&01&01&01\\
\hline 00&20&02&20&00&00\\
00&22&20&02&00&00\\
20&02&20&00&00&00\\
02&22&02&00&00&00\\
\end{array} \right).$$
\caption{Lengthenings of type $\ZZ_4^3\ZZ_2^2$ and type $\ZZ_4^2\ZZ_2^4$}\label{t:42}
\end{table}

\begin{proposition}\label{p:doob11}
In Doob graphs of diameter~$11$, there are no additive codes of size~$1024$ with weights $6$, $8$, $10$.
\end{proposition}
\begin{corollary}
 In Doob graphs $D(m,11-2m)$, $m>0$,
 there is no additive completely regular code~$C$ with intersection array
 $\{ 33, 30, 15; 1, 2, 15 \} $.
\end{corollary}
\begin{proof}
 Since the given intersection array is the same as that of the punctured
 dodecacode~$D^-$ in $H(11,4)$ and since the graphs $D(m,11-2m)$ and $H(11,4)$ have the same parameters as distance-regular graphs, the weight distributions of~$C$ and~$D^-$ coincide,
 see, e.g., \cite[Proposition 1.21]{KroPot:CRC&EP}. The weight distributions of the dual codes~$C^\perp$ (with respect to the inner product~\eqref{eq:ip})
 and~$D^{-\perp}$ also coincide by the MacWilliams transform~\cite[Theorem~6.6]{BesKro:Ch6}.
\end{proof}
\begin{proposition}\label{p:doob12}
In Doob graphs of diameter~$12$, there are no additive codes of size~$4096$ with weights $6$, $8$, $10$, $12$.
\end{proposition}
\begin{proof}
For the graph $D(m,12-2m)$, $0<m<6$, the claim follows from Proposition~\ref{p:doob11} (indeed, if we have a code with required parameters, then shortening it in the last coordinate leads to a  code of size~$1024$ in $D(m,11-2m)$, which does not exist).

For the case $m=6$,
we consider an additive code of type
$\ZZ_4^\delta\ZZ_2^\gamma $ in $D(6,0)$, with weights
$6$, $8$, $10$, $12$. The type-$\ZZ_2^{\gamma+\delta}$ subcode has Doob distance at least~$6$, and it is equivalent to an additive $(6,2^{\gamma+\delta},3)_4$ code, a $4$-ary code of length~$6$. In its turn, by concatenation (mapping the $4$ quaternary symbols to the binary triples $000$, $011$, $101$, $110$), each of such codes can be mapped to a binary linear $(18,2^{k},\{6,8,10,12\})$ code, $k=\gamma+\delta$. Such codes are classified with the software \cite{QextNewEdition}:
there are $2859$, $258$, and $3$ equivalence classes of such codes for $k=6,7,8$, respectively.
By inverse concatenation, we get $646$ equivalence classes of additive $(6,2^{6},3)_4$ codes
(including one distance-$4$ code, known as the hexacode), see Appendix~\ref{a:d3k6},
$14$ equivalence classes of additive $(6,2^{7},3)_4$ codes, see Appendix~\ref{a:d3k7}, and no additive $(6,2^{8},3)_4$ codes
(only $626$ of the $2859$ binary $(18,2^{6},\{6,8,10,12\})$ codes can be represented as quaternary by concatenation; on the other hand, some codes can be represented in more than one way;
that is why we say the number of equivalence classes of quaternary codes is larger than $626$).
So, now we know all additive distance-$6$ codes of type $\ZZ_2^k$, $k\ge 6$ in $D(6,0)$.
To confirm the results, we have made a search similar to that in Section~\ref{s:part},
see Table~\ref{t:d3}.

\begin{table}
 \begin{alignat*}{5}
\begin{array}{c|*{13}{c}}
  n\backslash k &0 &1 & 2 & 3 & 4 & 5 & 6 &7 &8 \\\hline 
   3  & 1 & 1 & 1  & 0 \\
   4  & 1 & 2 & 5  & 3  & 1  & 0  \\
   5  & 1 & 3 & 14 & 32  & 40    & 9    & 1   & 0                \\
   6  & 1 & 4 & 30 & 181 & 885   & 1660 & 646 & 14          & 0 & \\
\end{array}
\end{alignat*}
\caption{The number $N(n,k,3)$ of equivalence classes of additive $(n,2^k,{\ge} 3)_4$ codes.}\label{t:d3}
\end{table}

In the case $k=6$, we have $\gamma=0$ and $\delta=6$,
and for a putative additive code $C$ of type
$\ZZ_4^6$ in $D(6,0)$ its type-$\ZZ_4^2$ subcode is exactly $2C$. So, we can start with a generator matrix
of $2C$ (which comes from a binary generator matrix of an additive $(6,2^{6},3)_4$ code by replacing $1$s by $2$s) and try to lift it row-by-row.
From the $646$ additive $(6,2^{6},3)_4$ codes we found, there are only two codes such that any row of the generated matrix (but not all six rows) can be lifted, with (additive) generator matrices
$$
\left(
\begin{array}{c@{\,}cc@{\,}cc@{\,}cc@{\,}cc@{\,}cc@{\,}c}
   0&1& 0&0& 1&0& 0&0& 0&0& 1&0 \\
   0&0& 0&0& 0&0& 1&0& 1&0& 1&0 \\
   0&0& 0&0& 0&1& 0&1& 0&0& 0&1 \\
   0&0& 0&1& 0&0& 0&1& 0&1& 0&0 \\
   1&0& 0&0& 1&1& 0&1& 0&1& 1&0 \\
   0&0& 1&0& 1&0& 0&0& 1&0& 1&0
\end{array}\right),
\qquad
\left(\begin{array}{c@{\,}cc@{\,}cc@{\,}cc@{\,}cc@{\,}cc@{\,}c}
   1&0& 0&1& 0&1& 0&1& 0&1& 1&1 \\
   0&0& 1&0& 0&1& 0&1& 0&0& 0&1 \\
   0&0& 0&1& 1&0& 0&0& 0&1& 0&1 \\
   0&0& 0&0& 0&1& 1&0& 0&1& 1&0 \\
   0&0& 0&1& 0&0& 0&1& 1&0& 1&0 \\
   0&1& 0&0& 0&0& 0&1& 0&1& 0&1
\end{array}\right)
$$
(the second code has distance~$4$ and is equivalent to the well-known hexacode).
By lifting row-by-row (and keeping only inequivalent codes at each step), we find that for each of the two matrices, it is not possible to lift the first $3$ rows to produce a code of type $\ZZ_4^3\ZZ_2^3$ in $D(6,0)$. It follows that a distance-$6$ even-weight code of type $\ZZ_4^6$ does not exist in~$D(6,0)$.

For the type $\ZZ_4^5\ZZ_2^2$, it is slightly more complicated. The reason is that $2C$ is of type $\ZZ_2^5$, which is not the whole type-$\ZZ_2^7$ subcode of~$C$. So, if we have a putative type-$\ZZ_2^7$ subcode
(from a $4$-ary additive $(6,2^{7},3)_4$ code)
where the rows of some concrete generator matrix cannot be lifted, this does not lead to a direct contradiction.
The strategy is to try to lift each nonzero codeword of the putative type-$\ZZ_2^7$ subcode.
For each of the $14$ candidates $C$ to such a subcode, obtained from the $14$ additive $[6,3.5,3]_4$ codes by replacing $1$s by $2$s,
we try to lift each nonzero codeword. A codeword $b$
is called \emph{liftable} if there is another word $c$ such that $b-2c$
and the additive code obtained as the additive closure of $C\cup\{c\}$
has only (Doob) weights $0$, $6$, $8$, $10$, $12$ (not necessarily all of them). Clearly, if there is a type-$\ZZ_4^5\ZZ_2^2$ code with required weights, then its type-$\ZZ_2^7$ subcode must have at least $31 = 2^5-1$ liftable codewords. However, among the $14$ candidates on such a subcode,
$4$ codes have only $7$ liftable codewords and $10$ codes have only $1$ liftable codeword. This means that a required code does not exist.
\end{proof}

\section{Conclusion and open problems}
We have shown that the punctured dodecacode is characterized by its weight distribution, and that it is not $\F_4$-linear. It is the case $m=2$ of an infinite family of putative uniformly packed codes introduced in \cite{BZZ:1974:UPC}. The case $m=3$ of that family would lead to consider a dual code $D$ of binary dimension $2(2m+1)=2\times 7=14$, and length $43.$
Its weight distribution would be, in Magma notation:
$$
\texttt{[ <0, 1>, <28, 3612>, <32, 8127>, <36, 4644> ]}.
$$
While this code would still be even and trace Hermitian self-orthogonal, the classification of self-orthogonal codes either for the Hermitian or trace Hermitian codes are far beyond the tables in \cite{BouOst:2005} and \cite{DanPar:2006}. A computer search for an analogue of the dodecacode in length~$44$ with a cyclic automorphism group has failed to produce an example.

\section*{Acknowledgments}
The work of Markus Grassl is carried out under the `International Centre for Theory of Quantum Technologies 2.0: R\&D Industrial and Experimental Phase' project (contract no.~FENG.02.01-IP.05-0006/23). The project is implemented as part of the International Research Agendas Programme of the Foundation for Polish Science, co-financed by the European Funds for a Smart Economy 2021-2027 (FENG), Priority FENG.02 Innovation-friendly environment, Measure FENG.02.01.
The work of D.~Krotov is supported  within the framework of the state contract of the Sobolev Institute of Mathematics (FWNF-2026-0011).
Lin Sok is supported by Nanyang Technological University Research Grant No. 04INS000047C230GRT01.


\providecommand\href[2]{#2} \providecommand\url[1]{\href{#1}{#1}}
  \def\DOI#1{{\href{https://doi.org/#1}{https://doi.org/#1}}}\def\DOIURL#1#2{{\href{https://doi.org/#2}{https://doi.org/#1}}}

\raggedbottom
\appendix
\section{\texorpdfstring{All weight-$\{6,8\}$ codes in $D(m,n)$ of diameter $9$}{All weight-\{6,8\} codes in D(m,n) of diameter 9}}
\label{s:all}

Here, we list generator matrices of all, up to equivalence, weight-$\{6,8\}$ codes in $D(m,n'+n'')$, $2m+n'+n''=9$, $2n+n''\ge 6$.
In a matrix, two vertical lines in separate the three groups of coordinates, of size $m$, $n'$, and $n''$, while the horizontal line separates the rows of order $4$ and rows of order-$2$.

\subsection{\texorpdfstring{In $D(4,1+0)$}{In D(4,1+0)}}
\begin{small}\raggedright

Size-$2$ codes of type $\ZZ_4^0\ZZ_2^1$:\\
$ 1:$~$\left( 
\begin{array}{@{}c@{\ }c@{\ }c@{\ }c@{\ }|@{\ }c@{\ }|@{}}
\hline 20&20&20&00&00\\
\end{array} \right),$\quad
$ 2:$~$\left( \begin{array}{@{}c@{\ }c@{\ }c@{\ }c@{\ }|@{\ }c@{\ }|@{}}
\hline 20&20&20&20&00\\
\end{array} \right).$

Size-$4$ codes of type $\ZZ_4^0\ZZ_2^2$:\\
$ 1:$~$\left( \begin{array}{@{}c@{\ }c@{\ }c@{\ }c@{\ }|@{\ }c@{\ }|@{}}
\hline 20&20&20&00&00\\
02&02&02&00&00\\
\end{array} \right),$\quad
$ 2:$~$\left( \begin{array}{@{}c@{\ }c@{\ }c@{\ }c@{\ }|@{\ }c@{\ }|@{}}
\hline 20&20&20&00&00\\
02&20&00&20&00\\
\end{array} \right),$\quad
$ 3:$~$\left( \begin{array}{@{}c@{\ }c@{\ }c@{\ }c@{\ }|@{\ }c@{\ }|@{}}
\hline 20&20&20&00&00\\
02&02&00&20&00\\
\end{array} \right),$\quad
$ 4:$~$\left( \begin{array}{@{}c@{\ }c@{\ }c@{\ }c@{\ }|@{\ }c@{\ }|@{}}
\hline 20&20&20&00&00\\
02&02&02&20&00\\
\end{array} \right),$\quad
$ 5:$~$\left( \begin{array}{@{}c@{\ }c@{\ }c@{\ }c@{\ }|@{\ }c@{\ }|@{}}
\hline 20&20&20&20&00\\
02&02&02&02&00\\
\end{array} \right).$

Size-$4$ codes of type $\ZZ_4^1\ZZ_2^0$:\\
$ 1:$~$\left( \begin{array}{@{}c@{\ }c@{\ }c@{\ }c@{\ }|@{\ }c@{\ }|@{}}
21&21&10&00&10\\
\hline \end{array} \right),$\quad
$ 2:$~$\left( \begin{array}{@{}c@{\ }c@{\ }c@{\ }c@{\ }|@{\ }c@{\ }|@{}}
21&21&21&00&00\\
\hline \end{array} \right),$\quad
$ 3:$~$\left( \begin{array}{@{}c@{\ }c@{\ }c@{\ }c@{\ }|@{\ }c@{\ }|@{}}
20&10&10&10&10\\
\hline \end{array} \right),$\quad
$ 4:$~$\left( \begin{array}{@{}c@{\ }c@{\ }c@{\ }c@{\ }|@{\ }c@{\ }|@{}}
21&10&10&10&10\\
\hline \end{array} \right),$\quad
$ 5:$~$\left( \begin{array}{@{}c@{\ }c@{\ }c@{\ }c@{\ }|@{\ }c@{\ }|@{}}
21&20&10&10&00\\
\hline \end{array} \right),$\quad
$ 6:$~$\left( \begin{array}{@{}c@{\ }c@{\ }c@{\ }c@{\ }|@{\ }c@{\ }|@{}}
21&21&10&10&00\\
\hline \end{array} \right),$\quad
$ 7:$~$\left( \begin{array}{@{}c@{\ }c@{\ }c@{\ }c@{\ }|@{\ }c@{\ }|@{}}
21&21&20&10&10\\
\hline \end{array} \right),$\quad
$ 8:$~$\left( \begin{array}{@{}c@{\ }c@{\ }c@{\ }c@{\ }|@{\ }c@{\ }|@{}}
21&21&21&10&10\\
\hline \end{array} \right),$\quad
$ 9:$~$\left( \begin{array}{@{}c@{\ }c@{\ }c@{\ }c@{\ }|@{\ }c@{\ }|@{}}
21&21&21&20&00\\
\hline \end{array} \right),$\quad
$ 10:$~$\left( \begin{array}{@{}c@{\ }c@{\ }c@{\ }c@{\ }|@{\ }c@{\ }|@{}}
21&21&21&21&00\\
\hline \end{array} \right).$

Size-$8$ codes of type $\ZZ_4^0\ZZ_2^3$:\\
$ 1:$~$\left( \begin{array}{@{}c@{\ }c@{\ }c@{\ }c@{\ }|@{\ }c@{\ }|@{}}
\hline 20&20&20&00&00\\
02&02&02&00&00\\
02&20&00&20&00\\
\end{array} \right),$\quad
$ 2:$~$\left( \begin{array}{@{}c@{\ }c@{\ }c@{\ }c@{\ }|@{\ }c@{\ }|@{}}
\hline 20&20&20&00&00\\
02&20&00&20&00\\
00&02&02&02&00\\
\end{array} \right),$\quad
$ 3:$~$\left( \begin{array}{@{}c@{\ }c@{\ }c@{\ }c@{\ }|@{\ }c@{\ }|@{}}
\hline 20&20&20&00&00\\
02&02&00&20&00\\
22&00&02&02&00\\
\end{array} \right).$

Size-$8$ codes of type $\ZZ_4^1\ZZ_2^1$:\\
$ 1:$~$\left( \begin{array}{@{}c@{\ }c@{\ }c@{\ }c@{\ }|@{\ }c@{\ }|@{}}
21&21&10&00&10\\
\hline 20&20&00&20&00\\
\end{array} \right),$\quad
$ 2:$~$\left( \begin{array}{@{}c@{\ }c@{\ }c@{\ }c@{\ }|@{\ }c@{\ }|@{}}
21&21&10&00&10\\
\hline 20&00&02&20&00\\
\end{array} \right),$\quad
$ 3:$~$\left( \begin{array}{@{}c@{\ }c@{\ }c@{\ }c@{\ }|@{\ }c@{\ }|@{}}
21&21&21&00&00\\
\hline 20&20&00&20&00\\
\end{array} \right),$\quad
$ 4:$~$\left( \begin{array}{@{}c@{\ }c@{\ }c@{\ }c@{\ }|@{\ }c@{\ }|@{}}
20&10&10&10&10\\
\hline 02&02&02&00&00\\
\end{array} \right),$\quad
$ 5:$~$\left( \begin{array}{@{}c@{\ }c@{\ }c@{\ }c@{\ }|@{\ }c@{\ }|@{}}
21&10&10&10&10\\
\hline 20&02&20&00&00\\
\end{array} \right),$\quad
$ 6:$~$\left( \begin{array}{@{}c@{\ }c@{\ }c@{\ }c@{\ }|@{\ }c@{\ }|@{}}
21&10&10&10&10\\
\hline 02&02&02&00&00\\
\end{array} \right),$\quad
$ 7:$~$\left( \begin{array}{@{}c@{\ }c@{\ }c@{\ }c@{\ }|@{\ }c@{\ }|@{}}
21&10&10&10&10\\
\hline 20&02&02&02&00\\
\end{array} \right),$\quad
$ 8:$~$\left( \begin{array}{@{}c@{\ }c@{\ }c@{\ }c@{\ }|@{\ }c@{\ }|@{}}
21&20&10&10&00\\
\hline 20&02&02&00&00\\
\end{array} \right),$\quad
$ 9:$~$\left( \begin{array}{@{}c@{\ }c@{\ }c@{\ }c@{\ }|@{\ }c@{\ }|@{}}
21&20&10&10&00\\
\hline 00&02&02&02&00\\
\end{array} \right),$\quad
$ 10:$~$\left( \begin{array}{@{}c@{\ }c@{\ }c@{\ }c@{\ }|@{\ }c@{\ }|@{}}
21&21&10&10&00\\
\hline 02&20&02&00&00\\
\end{array} \right),$\quad
$ 11:$~$\left( \begin{array}{@{}c@{\ }c@{\ }c@{\ }c@{\ }|@{\ }c@{\ }|@{}}
21&21&10&10&00\\
\hline 02&00&02&02&00\\
\end{array} \right),$\quad
$ 12:$~$\left( \begin{array}{@{}c@{\ }c@{\ }c@{\ }c@{\ }|@{\ }c@{\ }|@{}}
21&21&10&10&00\\
\hline 20&20&02&02&00\\
\end{array} \right),$\quad
$ 13:$~$\left( \begin{array}{@{}c@{\ }c@{\ }c@{\ }c@{\ }|@{\ }c@{\ }|@{}}
21&21&20&10&10\\
\hline 20&00&02&02&00\\
\end{array} \right),$\quad
$ 14:$~$\left( \begin{array}{@{}c@{\ }c@{\ }c@{\ }c@{\ }|@{\ }c@{\ }|@{}}
21&21&21&10&10\\
\hline 02&20&00&02&00\\
\end{array} \right).$

Size-$16$ codes of type $\ZZ_4^0\ZZ_2^4$:\\
$ 1:$~$\left( \begin{array}{@{}c@{\ }c@{\ }c@{\ }c@{\ }|@{\ }c@{\ }|@{}}
\hline 20&20&20&00&00\\
02&02&02&00&00\\
02&20&00&20&00\\
22&02&00&02&00\\
\end{array} \right).$

Size-$16$ codes of type $\ZZ_4^1\ZZ_2^2$:\\
$ 1:$~$\left( \begin{array}{@{}c@{\ }c@{\ }c@{\ }c@{\ }|@{\ }c@{\ }|@{}}
21&21&10&00&10\\
\hline 20&20&00&20&00\\
20&00&02&02&00\\
\end{array} \right),$\quad
$ 2:$~$\left( \begin{array}{@{}c@{\ }c@{\ }c@{\ }c@{\ }|@{\ }c@{\ }|@{}}
21&21&10&00&10\\
\hline 20&20&00&20&00\\
22&00&02&02&00\\
\end{array} \right),$\quad
$ 3:$~$\left( \begin{array}{@{}c@{\ }c@{\ }c@{\ }c@{\ }|@{\ }c@{\ }|@{}}
21&21&21&00&00\\
\hline 20&20&00&20&00\\
20&00&20&02&00\\
\end{array} \right),$\quad
$ 4:$~$\left( \begin{array}{@{}c@{\ }c@{\ }c@{\ }c@{\ }|@{\ }c@{\ }|@{}}
21&21&21&00&00\\
\hline 20&20&00&20&00\\
22&00&20&02&00\\
\end{array} \right),$\quad
$ 5:$~$\left( \begin{array}{@{}c@{\ }c@{\ }c@{\ }c@{\ }|@{\ }c@{\ }|@{}}
21&10&10&10&10\\
\hline 20&02&20&00&00\\
22&20&02&00&00\\
\end{array} \right),$\quad
$ 6:$~$\left( \begin{array}{@{}c@{\ }c@{\ }c@{\ }c@{\ }|@{\ }c@{\ }|@{}}
21&10&10&10&10\\
\hline 20&02&20&00&00\\
02&00&02&02&00\\
\end{array} \right),$\quad
$ 7:$~$\left( \begin{array}{@{}c@{\ }c@{\ }c@{\ }c@{\ }|@{\ }c@{\ }|@{}}
21&10&10&10&10\\
\hline 02&02&02&00&00\\
02&22&00&02&00\\
\end{array} \right),$\quad
$ 8:$~$\left( \begin{array}{@{}c@{\ }c@{\ }c@{\ }c@{\ }|@{\ }c@{\ }|@{}}
21&21&10&10&00\\
\hline 02&20&02&00&00\\
20&02&00&02&00\\
\end{array} \right),$\quad
$ 9:$~$\left( \begin{array}{@{}c@{\ }c@{\ }c@{\ }c@{\ }|@{\ }c@{\ }|@{}}
21&21&10&10&00\\
\hline 02&20&02&00&00\\
02&22&00&02&00\\
\end{array} \right).$

Size-$16$ codes of type $\ZZ_4^2\ZZ_2^0$:\\
$ 1:$~$\left( \begin{array}{@{}c@{\ }c@{\ }c@{\ }c@{\ }|@{\ }c@{\ }|@{}}
21&21&10&00&10\\
12&31&00&10&01\\
\hline \end{array} \right),$\quad
$ 2:$~$\left( \begin{array}{@{}c@{\ }c@{\ }c@{\ }c@{\ }|@{\ }c@{\ }|@{}}
21&21&10&00&10\\
20&10&01&10&01\\
\hline \end{array} \right),$\quad
$ 3:$~$\left( \begin{array}{@{}c@{\ }c@{\ }c@{\ }c@{\ }|@{\ }c@{\ }|@{}}
21&21&10&00&10\\
02&10&01&10&10\\
\hline \end{array} \right),$\quad
$ 4:$~$\left( \begin{array}{@{}c@{\ }c@{\ }c@{\ }c@{\ }|@{\ }c@{\ }|@{}}
21&21&10&00&10\\
31&20&01&10&00\\
\hline \end{array} \right),$\quad
$ 5:$~$\left( \begin{array}{@{}c@{\ }c@{\ }c@{\ }c@{\ }|@{\ }c@{\ }|@{}}
21&21&10&00&10\\
21&30&01&10&01\\
\hline \end{array} \right),$\quad
$ 6:$~$\left( \begin{array}{@{}c@{\ }c@{\ }c@{\ }c@{\ }|@{\ }c@{\ }|@{}}
21&21&10&00&10\\
10&10&02&10&01\\
\hline \end{array} \right),$\quad
$ 7:$~$\left( \begin{array}{@{}c@{\ }c@{\ }c@{\ }c@{\ }|@{\ }c@{\ }|@{}}
21&21&10&00&10\\
30&10&02&10&10\\
\hline \end{array} \right),$\quad
$ 8:$~$\left( \begin{array}{@{}c@{\ }c@{\ }c@{\ }c@{\ }|@{\ }c@{\ }|@{}}
21&21&10&00&10\\
31&10&02&10&00\\
\hline \end{array} \right),$\quad
$ 9:$~$\left( \begin{array}{@{}c@{\ }c@{\ }c@{\ }c@{\ }|@{\ }c@{\ }|@{}}
21&21&10&00&10\\
31&31&00&21&00\\
\hline \end{array} \right),$\quad
$ 10:$~$\left( \begin{array}{@{}c@{\ }c@{\ }c@{\ }c@{\ }|@{\ }c@{\ }|@{}}
21&21&10&00&10\\
31&00&01&21&01\\
\hline \end{array} \right),$\quad
$ 11:$~$\left( \begin{array}{@{}c@{\ }c@{\ }c@{\ }c@{\ }|@{\ }c@{\ }|@{}}
21&21&10&00&10\\
12&00&01&21&10\\
\hline \end{array} \right),$\quad
$ 12:$~$\left( \begin{array}{@{}c@{\ }c@{\ }c@{\ }c@{\ }|@{\ }c@{\ }|@{}}
21&21&10&00&10\\
01&10&01&21&01\\
\hline \end{array} \right),$\quad
$ 13:$~$\left( \begin{array}{@{}c@{\ }c@{\ }c@{\ }c@{\ }|@{\ }c@{\ }|@{}}
21&21&10&00&10\\
21&10&01&21&00\\
\hline \end{array} \right),$\quad
$ 14:$~$\left( \begin{array}{@{}c@{\ }c@{\ }c@{\ }c@{\ }|@{\ }c@{\ }|@{}}
21&21&10&00&10\\
02&10&01&21&00\\
\hline \end{array} \right),$\quad
$ 15:$~$\left( \begin{array}{@{}c@{\ }c@{\ }c@{\ }c@{\ }|@{\ }c@{\ }|@{}}
21&21&10&00&10\\
30&20&01&21&00\\
\hline \end{array} \right),$\quad
$ 16:$~$\left( \begin{array}{@{}c@{\ }c@{\ }c@{\ }c@{\ }|@{\ }c@{\ }|@{}}
21&21&10&00&10\\
31&20&01&21&10\\
\hline \end{array} \right),$\quad
$ 17:$~$\left( \begin{array}{@{}c@{\ }c@{\ }c@{\ }c@{\ }|@{\ }c@{\ }|@{}}
21&21&10&00&10\\
31&21&01&21&01\\
\hline \end{array} \right),$\quad
$ 18:$~$\left( \begin{array}{@{}c@{\ }c@{\ }c@{\ }c@{\ }|@{\ }c@{\ }|@{}}
21&21&10&00&10\\
30&10&02&21&00\\
\hline \end{array} \right),$\quad
$ 19:$~$\left( \begin{array}{@{}c@{\ }c@{\ }c@{\ }c@{\ }|@{\ }c@{\ }|@{}}
21&21&10&00&10\\
31&10&02&21&10\\
\hline \end{array} \right),$\quad
$ 20:$~$\left( \begin{array}{@{}c@{\ }c@{\ }c@{\ }c@{\ }|@{\ }c@{\ }|@{}}
21&21&10&00&10\\
12&10&02&21&01\\
\hline \end{array} \right),$\quad
$ 21:$~$\left( \begin{array}{@{}c@{\ }c@{\ }c@{\ }c@{\ }|@{\ }c@{\ }|@{}}
21&21&21&00&00\\
30&20&10&10&10\\
\hline \end{array} \right),$\quad
$ 22:$~$\left( \begin{array}{@{}c@{\ }c@{\ }c@{\ }c@{\ }|@{\ }c@{\ }|@{}}
21&21&21&00&00\\
31&31&00&21&00\\
\hline \end{array} \right),$\quad
$ 23:$~$\left( \begin{array}{@{}c@{\ }c@{\ }c@{\ }c@{\ }|@{\ }c@{\ }|@{}}
21&21&21&00&00\\
30&20&10&21&00\\
\hline \end{array} \right),$\quad
$ 24:$~$\left( \begin{array}{@{}c@{\ }c@{\ }c@{\ }c@{\ }|@{\ }c@{\ }|@{}}
21&21&21&00&00\\
31&20&10&21&10\\
\hline \end{array} \right),$\quad
$ 25:$~$\left( \begin{array}{@{}c@{\ }c@{\ }c@{\ }c@{\ }|@{\ }c@{\ }|@{}}
20&10&10&10&10\\
01&02&01&01&10\\
\hline \end{array} \right),$\quad
$ 26:$~$\left( \begin{array}{@{}c@{\ }c@{\ }c@{\ }c@{\ }|@{\ }c@{\ }|@{}}
20&10&10&10&10\\
21&02&01&01&00\\
\hline \end{array} \right),$\quad
$ 27:$~$\left( \begin{array}{@{}c@{\ }c@{\ }c@{\ }c@{\ }|@{\ }c@{\ }|@{}}
20&10&10&10&10\\
01&02&11&01&01\\
\hline \end{array} \right),$\quad
$ 28:$~$\left( \begin{array}{@{}c@{\ }c@{\ }c@{\ }c@{\ }|@{\ }c@{\ }|@{}}
20&10&10&10&10\\
01&12&11&01&01\\
\hline \end{array} \right),$\quad
$ 29:$~$\left( \begin{array}{@{}c@{\ }c@{\ }c@{\ }c@{\ }|@{\ }c@{\ }|@{}}
20&10&10&10&10\\
01&02&21&01&00\\
\hline \end{array} \right),$\quad
$ 30:$~$\left( \begin{array}{@{}c@{\ }c@{\ }c@{\ }c@{\ }|@{\ }c@{\ }|@{}}
20&10&10&10&10\\
21&02&21&01&10\\
\hline \end{array} \right),$\quad
$ 31:$~$\left( \begin{array}{@{}c@{\ }c@{\ }c@{\ }c@{\ }|@{\ }c@{\ }|@{}}
21&10&10&10&10\\
31&01&01&01&01\\
\hline \end{array} \right),$\quad
$ 32:$~$\left( \begin{array}{@{}c@{\ }c@{\ }c@{\ }c@{\ }|@{\ }c@{\ }|@{}}
21&10&10&10&10\\
31&11&01&01&10\\
\hline \end{array} \right),$\quad
$ 33:$~$\left( \begin{array}{@{}c@{\ }c@{\ }c@{\ }c@{\ }|@{\ }c@{\ }|@{}}
21&10&10&10&10\\
12&11&01&01&01\\
\hline \end{array} \right),$\quad
$ 34:$~$\left( \begin{array}{@{}c@{\ }c@{\ }c@{\ }c@{\ }|@{\ }c@{\ }|@{}}
21&10&10&10&10\\
10&21&01&01&10\\
\hline \end{array} \right),$\quad
$ 35:$~$\left( \begin{array}{@{}c@{\ }c@{\ }c@{\ }c@{\ }|@{\ }c@{\ }|@{}}
21&10&10&10&10\\
10&21&11&01&01\\
\hline \end{array} \right),$\quad
$ 36:$~$\left( \begin{array}{@{}c@{\ }c@{\ }c@{\ }c@{\ }|@{\ }c@{\ }|@{}}
21&10&10&10&10\\
30&21&11&01&10\\
\hline \end{array} \right),$\quad
$ 37:$~$\left( \begin{array}{@{}c@{\ }c@{\ }c@{\ }c@{\ }|@{\ }c@{\ }|@{}}
21&10&10&10&10\\
01&02&21&01&00\\
\hline \end{array} \right),$\quad
$ 38:$~$\left( \begin{array}{@{}c@{\ }c@{\ }c@{\ }c@{\ }|@{\ }c@{\ }|@{}}
21&10&10&10&10\\
10&12&02&01&00\\
\hline \end{array} \right),$\quad
$ 39:$~$\left( \begin{array}{@{}c@{\ }c@{\ }c@{\ }c@{\ }|@{\ }c@{\ }|@{}}
21&10&10&10&10\\
31&12&02&01&01\\
\hline \end{array} \right),$\quad
$ 40:$~$\left( \begin{array}{@{}c@{\ }c@{\ }c@{\ }c@{\ }|@{\ }c@{\ }|@{}}
21&20&10&10&00\\
20&01&21&01&00\\
\hline \end{array} \right),$\quad
$ 41:$~$\left( \begin{array}{@{}c@{\ }c@{\ }c@{\ }c@{\ }|@{\ }c@{\ }|@{}}
21&20&10&10&00\\
12&01&02&01&00\\
\hline \end{array} \right),$\quad
$ 42:$~$\left( \begin{array}{@{}c@{\ }c@{\ }c@{\ }c@{\ }|@{\ }c@{\ }|@{}}
21&21&10&10&00\\
31&31&01&01&00\\
\hline \end{array} \right),$\quad
$ 43:$~$\left( \begin{array}{@{}c@{\ }c@{\ }c@{\ }c@{\ }|@{\ }c@{\ }|@{}}
21&21&10&10&00\\
12&31&11&01&00\\
\hline \end{array} \right),$\quad
$ 44:$~$\left( \begin{array}{@{}c@{\ }c@{\ }c@{\ }c@{\ }|@{\ }c@{\ }|@{}}
21&21&10&10&00\\
12&10&21&01&00\\
\hline \end{array} \right).$

Size-$32$ codes of type $\ZZ_4^0\ZZ_2^5$:
none.

Size-$32$ codes of type $\ZZ_4^1\ZZ_2^3$:\\
$ 1:$~$\left( \begin{array}{@{}c@{\ }c@{\ }c@{\ }c@{\ }|@{\ }c@{\ }|@{}}
21&10&10&10&10\\
\hline 20&02&20&00&00\\
22&20&02&00&00\\
20&20&00&02&00\\
\end{array} \right).$

Size-$32$ codes of type $\ZZ_4^2\ZZ_2^1$:\\
$ 1:$~$\left( \begin{array}{@{}c@{\ }c@{\ }c@{\ }c@{\ }|@{\ }c@{\ }|@{}}
21&21&10&00&10\\
12&31&00&10&01\\
\hline 22&00&02&02&00\\
\end{array} \right),$\quad
$ 2:$~$\left( \begin{array}{@{}c@{\ }c@{\ }c@{\ }c@{\ }|@{\ }c@{\ }|@{}}
21&21&10&00&10\\
20&10&01&10&01\\
\hline 20&22&00&02&00\\
\end{array} \right),$\quad
$ 3:$~$\left( \begin{array}{@{}c@{\ }c@{\ }c@{\ }c@{\ }|@{\ }c@{\ }|@{}}
21&21&10&00&10\\
20&10&01&10&01\\
\hline 22&22&00&02&00\\
\end{array} \right),$\quad
$ 4:$~$\left( \begin{array}{@{}c@{\ }c@{\ }c@{\ }c@{\ }|@{\ }c@{\ }|@{}}
21&21&10&00&10\\
02&10&01&10&10\\
\hline 20&22&00&02&00\\
\end{array} \right),$\quad
$ 5:$~$\left( \begin{array}{@{}c@{\ }c@{\ }c@{\ }c@{\ }|@{\ }c@{\ }|@{}}
21&21&10&00&10\\
31&20&01&10&00\\
\hline 20&22&00&02&00\\
\end{array} \right),$\quad
$ 6:$~$\left( \begin{array}{@{}c@{\ }c@{\ }c@{\ }c@{\ }|@{\ }c@{\ }|@{}}
21&21&10&00&10\\
21&30&01&10&01\\
\hline 20&20&00&02&00\\
\end{array} \right),$\quad
$ 7:$~$\left( \begin{array}{@{}c@{\ }c@{\ }c@{\ }c@{\ }|@{\ }c@{\ }|@{}}
21&21&10&00&10\\
31&10&02&10&00\\
\hline 22&02&02&02&00\\
\end{array} \right),$\quad
$ 8:$~$\left( \begin{array}{@{}c@{\ }c@{\ }c@{\ }c@{\ }|@{\ }c@{\ }|@{}}
21&10&10&10&10\\
31&01&01&01&01\\
\hline 20&02&20&00&00\\
\end{array} \right),$\quad
$ 9:$~$\left( \begin{array}{@{}c@{\ }c@{\ }c@{\ }c@{\ }|@{\ }c@{\ }|@{}}
21&10&10&10&10\\
31&11&01&01&10\\
\hline 22&02&20&00&00\\
\end{array} \right),$\quad
$ 10:$~$\left( \begin{array}{@{}c@{\ }c@{\ }c@{\ }c@{\ }|@{\ }c@{\ }|@{}}
21&10&10&10&10\\
12&11&01&01&01\\
\hline 20&02&20&00&00\\
\end{array} \right),$\quad
$ 11:$~$\left( \begin{array}{@{}c@{\ }c@{\ }c@{\ }c@{\ }|@{\ }c@{\ }|@{}}
21&10&10&10&10\\
12&11&01&01&01\\
\hline 20&20&22&00&00\\
\end{array} \right),$\quad
$ 12:$~$\left( \begin{array}{@{}c@{\ }c@{\ }c@{\ }c@{\ }|@{\ }c@{\ }|@{}}
21&21&10&10&00\\
12&31&11&01&00\\
\hline 20&02&02&00&00\\
\end{array} \right).$

Size-$64$ codes of type $\ZZ_4^0\ZZ_2^6$:
none.

Size-$64$ codes of type $\ZZ_4^1\ZZ_2^4$:
none.

Size-$64$ codes of type $\ZZ_4^2\ZZ_2^2$:\\
$ 1:$~$\left( \begin{array}{@{}c@{\ }c@{\ }c@{\ }c@{\ }|@{\ }c@{\ }|@{}}
21&10&10&10&10\\
31&01&01&01&01\\
\hline 20&02&20&00&00\\
02&22&02&00&00\\
\end{array} \right),$\quad
$ 2:$~$\left( \begin{array}{@{}c@{\ }c@{\ }c@{\ }c@{\ }|@{\ }c@{\ }|@{}}
21&10&10&10&10\\
12&11&01&01&01\\
\hline 20&02&20&00&00\\
22&20&02&00&00\\
\end{array} \right).$

Size-$64$ codes of type $\ZZ_4^3\ZZ_2^0$:\\
$ 1:$~$\left( \begin{array}{@{}c@{\ }c@{\ }c@{\ }c@{\ }|@{\ }c@{\ }|@{}}
21&21&10&00&10\\
20&10&01&10&01\\
12&01&01&01&01\\
\hline \end{array} \right),$\quad
$ 2:$~$\left( \begin{array}{@{}c@{\ }c@{\ }c@{\ }c@{\ }|@{\ }c@{\ }|@{}}
21&21&10&00&10\\
20&10&01&10&01\\
13&01&01&01&11\\
\hline \end{array} \right),$\quad
$ 3:$~$\left( \begin{array}{@{}c@{\ }c@{\ }c@{\ }c@{\ }|@{\ }c@{\ }|@{}}
21&21&10&00&10\\
20&10&01&10&01\\
32&03&01&01&10\\
\hline \end{array} \right),$\quad
$ 4:$~$\left( \begin{array}{@{}c@{\ }c@{\ }c@{\ }c@{\ }|@{\ }c@{\ }|@{}}
21&21&10&00&10\\
31&20&01&10&00\\
21&33&01&01&01\\
\hline \end{array} \right).$
\end{small}

\subsection{\texorpdfstring{In $D(3,3+0)$}{In D(3,3+0)}}
\begin{small}\raggedright
Size-$2$ codes of type $\ZZ_4^0\ZZ_2^1$:\\
$ 1:$~$\left( \begin{array}{@{}c@{\ }c@{\ }c@{\ }|@{\ }c@{\ }c@{\ }c@{\ }|}
\hline 20&20&00&10&10&00\\
\end{array} \right),$\quad
$ 2:$~$\left( \begin{array}{@{}c@{\ }c@{\ }c@{\ }|@{\ }c@{\ }c@{\ }c@{\ }|}
\hline 20&20&20&00&00&00\\
\end{array} \right),$\quad
$ 3:$~$\left( \begin{array}{@{}c@{\ }c@{\ }c@{\ }|@{\ }c@{\ }c@{\ }c@{\ }|}
\hline 20&20&20&10&10&00\\
\end{array} \right).$

Size-$4$ codes of type $\ZZ_4^0\ZZ_2^2$:\\
$ 1:$~$\left( \begin{array}{@{}c@{\ }c@{\ }c@{\ }|@{\ }c@{\ }c@{\ }c@{\ }|}
\hline 20&20&00&10&10&00\\
02&02&00&01&01&00\\
\end{array} \right),$\quad
$ 2:$~$\left( \begin{array}{@{}c@{\ }c@{\ }c@{\ }|@{\ }c@{\ }c@{\ }c@{\ }|}
\hline 20&20&00&10&10&00\\
02&02&00&10&00&10\\
\end{array} \right),$\quad
$ 3:$~$\left( \begin{array}{@{}c@{\ }c@{\ }c@{\ }|@{\ }c@{\ }c@{\ }c@{\ }|}
\hline 20&20&00&10&10&00\\
20&00&20&01&01&00\\
\end{array} \right),$\quad
$ 4:$~$\left( \begin{array}{@{}c@{\ }c@{\ }c@{\ }|@{\ }c@{\ }c@{\ }c@{\ }|}
\hline 20&20&00&10&10&00\\
20&00&20&10&00&10\\
\end{array} \right),$\quad
$ 5:$~$\left( \begin{array}{@{}c@{\ }c@{\ }c@{\ }|@{\ }c@{\ }c@{\ }c@{\ }|}
\hline 20&20&00&10&10&00\\
02&00&20&10&10&00\\
\end{array} \right),$\quad
$ 6:$~$\left( \begin{array}{@{}c@{\ }c@{\ }c@{\ }|@{\ }c@{\ }c@{\ }c@{\ }|}
\hline 20&20&00&10&10&00\\
02&00&20&01&01&00\\
\end{array} \right),$\quad
$ 7:$~$\left( \begin{array}{@{}c@{\ }c@{\ }c@{\ }|@{\ }c@{\ }c@{\ }c@{\ }|}
\hline 20&20&00&10&10&00\\
02&00&20&10&00&10\\
\end{array} \right),$\quad
$ 8:$~$\left( \begin{array}{@{}c@{\ }c@{\ }c@{\ }|@{\ }c@{\ }c@{\ }c@{\ }|}
\hline 20&20&00&10&10&00\\
02&02&20&00&00&00\\
\end{array} \right),$\quad
$ 9:$~$\left( \begin{array}{@{}c@{\ }c@{\ }c@{\ }|@{\ }c@{\ }c@{\ }c@{\ }|}
\hline 20&20&00&10&10&00\\
02&02&20&01&01&00\\
\end{array} \right),$\quad
$ 10:$~$\left( \begin{array}{@{}c@{\ }c@{\ }c@{\ }|@{\ }c@{\ }c@{\ }c@{\ }|}
\hline 20&20&00&10&10&00\\
02&02&20&10&00&10\\
\end{array} \right),$\quad
$ 11:$~$\left( \begin{array}{@{}c@{\ }c@{\ }c@{\ }|@{\ }c@{\ }c@{\ }c@{\ }|}
\hline 20&20&20&00&00&00\\
02&02&02&00&00&00\\
\end{array} \right),$\quad
$ 12:$~$\left( \begin{array}{@{}c@{\ }c@{\ }c@{\ }|@{\ }c@{\ }c@{\ }c@{\ }|}
\hline 20&20&20&00&00&00\\
02&02&02&10&10&00\\
\end{array} \right),$\quad
$ 13:$~$\left( \begin{array}{@{}c@{\ }c@{\ }c@{\ }|@{\ }c@{\ }c@{\ }c@{\ }|}
\hline 20&20&20&10&10&00\\
02&02&02&01&01&00\\
\end{array} \right),$\quad
$ 14:$~$\left( \begin{array}{@{}c@{\ }c@{\ }c@{\ }|@{\ }c@{\ }c@{\ }c@{\ }|}
\hline 20&20&20&10&10&00\\
02&02&02&10&00&10\\
\end{array} \right).$

Size-$4$ codes of type $\ZZ_4^1\ZZ_2^0$:\\
$ 1:$~$\left( \begin{array}{@{}c@{\ }c@{\ }c@{\ }|@{\ }c@{\ }c@{\ }c@{\ }|}
10&10&10&10&10&10\\
\hline \end{array} \right),$\quad
$ 2:$~$\left( \begin{array}{@{}c@{\ }c@{\ }c@{\ }|@{\ }c@{\ }c@{\ }c@{\ }|}
21&10&10&10&10&00\\
\hline \end{array} \right),$\quad
$ 3:$~$\left( \begin{array}{@{}c@{\ }c@{\ }c@{\ }|@{\ }c@{\ }c@{\ }c@{\ }|}
21&21&10&10&00&00\\
\hline \end{array} \right),$\quad
$ 4:$~$\left( \begin{array}{@{}c@{\ }c@{\ }c@{\ }|@{\ }c@{\ }c@{\ }c@{\ }|}
21&21&10&10&10&10\\
\hline \end{array} \right),$\quad
$ 5:$~$\left( \begin{array}{@{}c@{\ }c@{\ }c@{\ }|@{\ }c@{\ }c@{\ }c@{\ }|}
21&21&21&00&00&00\\
\hline \end{array} \right),$\quad
$ 6:$~$\left( \begin{array}{@{}c@{\ }c@{\ }c@{\ }|@{\ }c@{\ }c@{\ }c@{\ }|}
21&21&21&10&10&00\\
\hline \end{array} \right).$

Size-$8$ codes of type $\ZZ_4^0\ZZ_2^3$:\\
$ 1:$~$\left( \begin{array}{@{}c@{\ }c@{\ }c@{\ }|@{\ }c@{\ }c@{\ }c@{\ }|}
\hline 20&20&00&10&10&00\\
02&02&00&01&01&00\\
20&00&20&01&01&00\\
\end{array} \right),$\quad
$ 2:$~$\left( \begin{array}{@{}c@{\ }c@{\ }c@{\ }|@{\ }c@{\ }c@{\ }c@{\ }|}
\hline 20&20&00&10&10&00\\
02&02&00&10&00&10\\
20&00&20&10&00&10\\
\end{array} \right),$\quad
$ 3:$~$\left( \begin{array}{@{}c@{\ }c@{\ }c@{\ }|@{\ }c@{\ }c@{\ }c@{\ }|}
\hline 20&20&00&10&10&00\\
20&00&20&01&01&00\\
02&02&02&00&00&00\\
\end{array} \right),$\quad
$ 4:$~$\left( \begin{array}{@{}c@{\ }c@{\ }c@{\ }|@{\ }c@{\ }c@{\ }c@{\ }|}
\hline 20&20&00&10&10&00\\
20&00&20&10&00&10\\
02&02&02&00&00&00\\
\end{array} \right),$\quad
$ 5:$~$\left( \begin{array}{@{}c@{\ }c@{\ }c@{\ }|@{\ }c@{\ }c@{\ }c@{\ }|}
\hline 20&20&00&10&10&00\\
02&00&20&10&10&00\\
00&02&02&10&10&00\\
\end{array} \right),$\quad
$ 6:$~$\left( \begin{array}{@{}c@{\ }c@{\ }c@{\ }|@{\ }c@{\ }c@{\ }c@{\ }|}
\hline 20&20&00&10&10&00\\
02&00&20&10&10&00\\
00&02&02&01&01&00\\
\end{array} \right),$\quad
$ 7:$~$\left( \begin{array}{@{}c@{\ }c@{\ }c@{\ }|@{\ }c@{\ }c@{\ }c@{\ }|}
\hline 20&20&00&10&10&00\\
02&00&20&10&10&00\\
00&02&02&10&00&10\\
\end{array} \right),$\quad
$ 8:$~$\left( \begin{array}{@{}c@{\ }c@{\ }c@{\ }|@{\ }c@{\ }c@{\ }c@{\ }|}
\hline 20&20&00&10&10&00\\
02&00&20&01&01&00\\
00&02&02&11&11&00\\
\end{array} \right),$\quad
$ 9:$~$\left( \begin{array}{@{}c@{\ }c@{\ }c@{\ }|@{\ }c@{\ }c@{\ }c@{\ }|}
\hline 20&20&00&10&10&00\\
02&00&20&10&00&10\\
00&02&02&00&10&10\\
\end{array} \right).$

Size-$8$ codes of type $\ZZ_4^1\ZZ_2^1$:\\
$ 1:$~$\left( \begin{array}{@{}c@{\ }c@{\ }c@{\ }|@{\ }c@{\ }c@{\ }c@{\ }|}
10&10&10&10&10&10\\
\hline 02&20&00&01&10&00\\
\end{array} \right),$\quad
$ 2:$~$\left( \begin{array}{@{}c@{\ }c@{\ }c@{\ }|@{\ }c@{\ }c@{\ }c@{\ }|}
10&10&10&10&10&10\\
\hline 02&02&00&10&10&00\\
\end{array} \right),$\quad
$ 3:$~$\left( \begin{array}{@{}c@{\ }c@{\ }c@{\ }|@{\ }c@{\ }c@{\ }c@{\ }|}
10&10&10&10&10&10\\
\hline 02&02&00&01&01&00\\
\end{array} \right),$\quad
$ 4:$~$\left( \begin{array}{@{}c@{\ }c@{\ }c@{\ }|@{\ }c@{\ }c@{\ }c@{\ }|}
10&10&10&10&10&10\\
\hline 02&02&02&01&10&00\\
\end{array} \right),$\quad
$ 5:$~$\left( \begin{array}{@{}c@{\ }c@{\ }c@{\ }|@{\ }c@{\ }c@{\ }c@{\ }|}
21&10&10&10&10&00\\
\hline 20&02&00&01&01&00\\
\end{array} \right),$\quad
$ 6:$~$\left( \begin{array}{@{}c@{\ }c@{\ }c@{\ }|@{\ }c@{\ }c@{\ }c@{\ }|}
21&10&10&10&10&00\\
\hline 20&02&00&10&00&10\\
\end{array} \right),$\quad
$ 7:$~$\left( \begin{array}{@{}c@{\ }c@{\ }c@{\ }|@{\ }c@{\ }c@{\ }c@{\ }|}
21&10&10&10&10&00\\
\hline 02&02&00&01&10&00\\
\end{array} \right),$\quad
$ 8:$~$\left( \begin{array}{@{}c@{\ }c@{\ }c@{\ }|@{\ }c@{\ }c@{\ }c@{\ }|}
21&10&10&10&10&00\\
\hline 02&02&00&01&00&10\\
\end{array} \right),$\quad
$ 9:$~$\left( \begin{array}{@{}c@{\ }c@{\ }c@{\ }|@{\ }c@{\ }c@{\ }c@{\ }|}
21&10&10&10&10&00\\
\hline 00&02&02&10&10&00\\
\end{array} \right),$\quad
$ 10:$~$\left( \begin{array}{@{}c@{\ }c@{\ }c@{\ }|@{\ }c@{\ }c@{\ }c@{\ }|}
21&10&10&10&10&00\\
\hline 00&02&02&01&01&00\\
\end{array} \right),$\quad
$ 11:$~$\left( \begin{array}{@{}c@{\ }c@{\ }c@{\ }|@{\ }c@{\ }c@{\ }c@{\ }|}
21&10&10&10&10&00\\
\hline 00&02&02&10&00&10\\
\end{array} \right),$\quad
$ 12:$~$\left( \begin{array}{@{}c@{\ }c@{\ }c@{\ }|@{\ }c@{\ }c@{\ }c@{\ }|}
21&10&10&10&10&00\\
\hline 20&02&02&01&10&00\\
\end{array} \right),$\quad
$ 13:$~$\left( \begin{array}{@{}c@{\ }c@{\ }c@{\ }|@{\ }c@{\ }c@{\ }c@{\ }|}
21&10&10&10&10&00\\
\hline 20&02&02&01&00&10\\
\end{array} \right),$\quad
$ 14:$~$\left( \begin{array}{@{}c@{\ }c@{\ }c@{\ }|@{\ }c@{\ }c@{\ }c@{\ }|}
21&21&10&10&00&00\\
\hline 20&00&02&10&10&00\\
\end{array} \right),$\quad
$ 15:$~$\left( \begin{array}{@{}c@{\ }c@{\ }c@{\ }|@{\ }c@{\ }c@{\ }c@{\ }|}
21&21&10&10&00&00\\
\hline 20&00&02&00&10&10\\
\end{array} \right),$\quad
$ 16:$~$\left( \begin{array}{@{}c@{\ }c@{\ }c@{\ }|@{\ }c@{\ }c@{\ }c@{\ }|}
21&21&10&10&00&00\\
\hline 02&00&02&01&10&00\\
\end{array} \right),$\quad
$ 17:$~$\left( \begin{array}{@{}c@{\ }c@{\ }c@{\ }|@{\ }c@{\ }c@{\ }c@{\ }|}
21&21&10&10&10&10\\
\hline 20&00&02&01&01&00\\
\end{array} \right),$\quad
$ 18:$~$\left( \begin{array}{@{}c@{\ }c@{\ }c@{\ }|@{\ }c@{\ }c@{\ }c@{\ }|}
21&21&10&10&10&10\\
\hline 02&00&02&01&10&00\\
\end{array} \right).$

Size-$16$ codes of type $\ZZ_4^0\ZZ_2^4$:\\
$ 1:$~$\left( \begin{array}{@{}c@{\ }c@{\ }c@{\ }|@{\ }c@{\ }c@{\ }c@{\ }|}
\hline 20&20&00&10&10&00\\
02&02&00&01&01&00\\
20&00&20&01&01&00\\
02&00&02&11&11&00\\
\end{array} \right),$\quad
$ 2:$~$\left( \begin{array}{@{}c@{\ }c@{\ }c@{\ }|@{\ }c@{\ }c@{\ }c@{\ }|}
\hline 20&20&00&10&10&00\\
02&02&00&10&00&10\\
20&00&20&10&00&10\\
02&00&02&00&10&10\\
\end{array} \right).$

Size-$16$ codes of type $\ZZ_4^1\ZZ_2^2$:\\
$ 1:$~$\left( \begin{array}{@{}c@{\ }c@{\ }c@{\ }|@{\ }c@{\ }c@{\ }c@{\ }|}
10&10&10&10&10&10\\
\hline 02&20&00&01&10&00\\
20&02&00&10&01&00\\
\end{array} \right),$\quad
$ 2:$~$\left( \begin{array}{@{}c@{\ }c@{\ }c@{\ }|@{\ }c@{\ }c@{\ }c@{\ }|}
10&10&10&10&10&10\\
\hline 02&20&00&01&10&00\\
20&02&00&01&00&10\\
\end{array} \right),$\quad
$ 3:$~$\left( \begin{array}{@{}c@{\ }c@{\ }c@{\ }|@{\ }c@{\ }c@{\ }c@{\ }|}
10&10&10&10&10&10\\
\hline 02&20&00&01&10&00\\
20&02&00&00&10&01\\
\end{array} \right),$\quad
$ 4:$~$\left( \begin{array}{@{}c@{\ }c@{\ }c@{\ }|@{\ }c@{\ }c@{\ }c@{\ }|}
10&10&10&10&10&10\\
\hline 02&20&00&01&10&00\\
00&02&02&11&01&00\\
\end{array} \right),$\quad
$ 5:$~$\left( \begin{array}{@{}c@{\ }c@{\ }c@{\ }|@{\ }c@{\ }c@{\ }c@{\ }|}
10&10&10&10&10&10\\
\hline 02&20&00&01&10&00\\
00&02&02&00&10&10\\
\end{array} \right),$\quad
$ 6:$~$\left( \begin{array}{@{}c@{\ }c@{\ }c@{\ }|@{\ }c@{\ }c@{\ }c@{\ }|}
10&10&10&10&10&10\\
\hline 02&20&00&01&10&00\\
00&02&02&01&00&01\\
\end{array} \right),$\quad
$ 7:$~$\left( \begin{array}{@{}c@{\ }c@{\ }c@{\ }|@{\ }c@{\ }c@{\ }c@{\ }|}
10&10&10&10&10&10\\
\hline 02&02&00&10&10&00\\
02&00&02&01&01&00\\
\end{array} \right),$\quad
$ 8:$~$\left( \begin{array}{@{}c@{\ }c@{\ }c@{\ }|@{\ }c@{\ }c@{\ }c@{\ }|}
10&10&10&10&10&10\\
\hline 02&02&00&10&10&00\\
02&00&02&10&00&10\\
\end{array} \right),$\quad
$ 9:$~$\left( \begin{array}{@{}c@{\ }c@{\ }c@{\ }|@{\ }c@{\ }c@{\ }c@{\ }|}
10&10&10&10&10&10\\
\hline 02&02&00&10&10&00\\
22&00&02&01&01&00\\
\end{array} \right),$\quad
$ 10:$~$\left( \begin{array}{@{}c@{\ }c@{\ }c@{\ }|@{\ }c@{\ }c@{\ }c@{\ }|}
10&10&10&10&10&10\\
\hline 02&02&00&10&10&00\\
22&00&02&10&00&10\\
\end{array} \right),$\quad
$ 11:$~$\left( \begin{array}{@{}c@{\ }c@{\ }c@{\ }|@{\ }c@{\ }c@{\ }c@{\ }|}
10&10&10&10&10&10\\
\hline 02&02&00&01&01&00\\
02&00&02&01&00&01\\
\end{array} \right),$\quad
$ 12:$~$\left( \begin{array}{@{}c@{\ }c@{\ }c@{\ }|@{\ }c@{\ }c@{\ }c@{\ }|}
10&10&10&10&10&10\\
\hline 02&02&00&01&01&00\\
22&00&02&01&00&01\\
\end{array} \right),$\quad
$ 13:$~$\left( \begin{array}{@{}c@{\ }c@{\ }c@{\ }|@{\ }c@{\ }c@{\ }c@{\ }|}
21&10&10&10&10&00\\
\hline 20&02&00&01&01&00\\
20&00&02&11&11&00\\
\end{array} \right),$\quad
$ 14:$~$\left( \begin{array}{@{}c@{\ }c@{\ }c@{\ }|@{\ }c@{\ }c@{\ }c@{\ }|}
21&10&10&10&10&00\\
\hline 20&02&00&01&01&00\\
02&00&02&11&10&00\\
\end{array} \right),$\quad
$ 15:$~$\left( \begin{array}{@{}c@{\ }c@{\ }c@{\ }|@{\ }c@{\ }c@{\ }c@{\ }|}
21&10&10&10&10&00\\
\hline 20&02&00&01&01&00\\
02&00&02&01&00&10\\
\end{array} \right),$\quad
$ 16:$~$\left( \begin{array}{@{}c@{\ }c@{\ }c@{\ }|@{\ }c@{\ }c@{\ }c@{\ }|}
21&10&10&10&10&00\\
\hline 20&02&00&01&01&00\\
22&00&02&11&11&00\\
\end{array} \right),$\quad
$ 17:$~$\left( \begin{array}{@{}c@{\ }c@{\ }c@{\ }|@{\ }c@{\ }c@{\ }c@{\ }|}
21&10&10&10&10&00\\
\hline 20&02&00&10&00&10\\
20&00&02&00&10&10\\
\end{array} \right),$\quad
$ 18:$~$\left( \begin{array}{@{}c@{\ }c@{\ }c@{\ }|@{\ }c@{\ }c@{\ }c@{\ }|}
21&10&10&10&10&00\\
\hline 20&02&00&10&00&10\\
02&00&02&10&01&00\\
\end{array} \right),$\quad
$ 19:$~$\left( \begin{array}{@{}c@{\ }c@{\ }c@{\ }|@{\ }c@{\ }c@{\ }c@{\ }|}
21&10&10&10&10&00\\
\hline 20&02&00&10&00&10\\
02&00&02&01&00&01\\
\end{array} \right),$\quad
$ 20:$~$\left( \begin{array}{@{}c@{\ }c@{\ }c@{\ }|@{\ }c@{\ }c@{\ }c@{\ }|}
21&10&10&10&10&00\\
\hline 20&02&00&10&00&10\\
22&00&02&00&10&10\\
\end{array} \right),$\quad
$ 21:$~$\left( \begin{array}{@{}c@{\ }c@{\ }c@{\ }|@{\ }c@{\ }c@{\ }c@{\ }|}
21&10&10&10&10&00\\
\hline 02&02&00&01&10&00\\
02&00&02&10&01&00\\
\end{array} \right),$\quad
$ 22:$~$\left( \begin{array}{@{}c@{\ }c@{\ }c@{\ }|@{\ }c@{\ }c@{\ }c@{\ }|}
21&10&10&10&10&00\\
\hline 02&02&00&01&10&00\\
02&00&02&01&00&10\\
\end{array} \right),$\quad
$ 23:$~$\left( \begin{array}{@{}c@{\ }c@{\ }c@{\ }|@{\ }c@{\ }c@{\ }c@{\ }|}
21&10&10&10&10&00\\
\hline 02&02&00&01&00&10\\
02&00&02&00&01&10\\
\end{array} \right),$\quad
$ 24:$~$\left( \begin{array}{@{}c@{\ }c@{\ }c@{\ }|@{\ }c@{\ }c@{\ }c@{\ }|}
21&10&10&10&10&00\\
\hline 02&02&00&01&00&10\\
02&00&02&11&00&01\\
\end{array} \right).$

Size-$16$ codes of type $\ZZ_4^2\ZZ_2^0$:
none.

Size-$32$ codes:
none.
\end{small}

\subsection{\texorpdfstring{In $D(4,0+1)$}{In D(4,0+1)}}
\begin{small}\raggedright

Size-$2$ codes of type $\ZZ_4^0\ZZ_2^1$:\\
$ 1:$~$\left( \begin{array}{@{}c@{\ }c@{\ }c@{\ }c@{\ }||@{\ }c@{}}
\hline 20&20&20&00&0\\
\end{array} \right),$\quad
$ 2:$~$\left( \begin{array}{@{}c@{\ }c@{\ }c@{\ }c@{\ }||@{\ }c@{}}
\hline 20&20&20&20&0\\
\end{array} \right).$

Size-$4$ codes of type $\ZZ_4^0\ZZ_2^2$:\\
$ 1:$~$\left( \begin{array}{@{}c@{\ }c@{\ }c@{\ }c@{\ }||@{\ }c@{}}
\hline 20&20&20&00&0\\
02&02&02&00&0\\
\end{array} \right),$\quad
$ 2:$~$\left( \begin{array}{@{}c@{\ }c@{\ }c@{\ }c@{\ }||@{\ }c@{}}
\hline 20&20&20&00&0\\
02&20&00&20&0\\
\end{array} \right),$\quad
$ 3:$~$\left( \begin{array}{@{}c@{\ }c@{\ }c@{\ }c@{\ }||@{\ }c@{}}
\hline 20&20&20&00&0\\
02&02&00&20&0\\
\end{array} \right),$\quad
$ 4:$~$\left( \begin{array}{@{}c@{\ }c@{\ }c@{\ }c@{\ }||@{\ }c@{}}
\hline 20&20&20&00&0\\
02&02&02&20&0\\
\end{array} \right),$\quad
$ 5:$~$\left( \begin{array}{@{}c@{\ }c@{\ }c@{\ }c@{\ }||@{\ }c@{}}
\hline 20&20&20&20&0\\
02&02&02&02&0\\
\end{array} \right).$

Size-$4$ codes of type $\ZZ_4^1\ZZ_2^0$:\\
$ 1:$~$\left( \begin{array}{@{}c@{\ }c@{\ }c@{\ }c@{\ }||@{\ }c@{}}
21&21&21&00&0\\
\hline \end{array} \right),$\quad
$ 2:$~$\left( \begin{array}{@{}c@{\ }c@{\ }c@{\ }c@{\ }||@{\ }c@{}}
21&20&10&10&0\\
\hline \end{array} \right),$\quad
$ 3:$~$\left( \begin{array}{@{}c@{\ }c@{\ }c@{\ }c@{\ }||@{\ }c@{}}
21&21&10&10&0\\
\hline \end{array} \right),$\quad
$ 4:$~$\left( \begin{array}{@{}c@{\ }c@{\ }c@{\ }c@{\ }||@{\ }c@{}}
21&21&21&20&0\\
\hline \end{array} \right),$\quad
$ 5:$~$\left( \begin{array}{@{}c@{\ }c@{\ }c@{\ }c@{\ }||@{\ }c@{}}
21&21&21&21&0\\
\hline \end{array} \right),$\quad
$ 6:$~$\left( \begin{array}{@{}c@{\ }c@{\ }c@{\ }c@{\ }||@{\ }c@{}}
21&21&10&00&2\\
\hline \end{array} \right),$\quad
$ 7:$~$\left( \begin{array}{@{}c@{\ }c@{\ }c@{\ }c@{\ }||@{\ }c@{}}
20&10&10&10&2\\
\hline \end{array} \right),$\quad
$ 8:$~$\left( \begin{array}{@{}c@{\ }c@{\ }c@{\ }c@{\ }||@{\ }c@{}}
21&10&10&10&2\\
\hline \end{array} \right),$\quad
$ 9:$~$\left( \begin{array}{@{}c@{\ }c@{\ }c@{\ }c@{\ }||@{\ }c@{}}
21&21&20&10&2\\
\hline \end{array} \right),$\quad
$ 10:$~$\left( \begin{array}{@{}c@{\ }c@{\ }c@{\ }c@{\ }||@{\ }c@{}}
21&21&21&10&2\\
\hline \end{array} \right).$

Size-$8$ codes of type $\ZZ_4^0\ZZ_2^3$:\\
$ 1:$~$\left( \begin{array}{@{}c@{\ }c@{\ }c@{\ }c@{\ }||@{\ }c@{}}
\hline 20&20&20&00&0\\
02&02&02&00&0\\
02&20&00&20&0\\
\end{array} \right),$\quad
$ 2:$~$\left( \begin{array}{@{}c@{\ }c@{\ }c@{\ }c@{\ }||@{\ }c@{}}
\hline 20&20&20&00&0\\
02&20&00&20&0\\
00&02&02&02&0\\
\end{array} \right),$\quad
$ 3:$~$\left( \begin{array}{@{}c@{\ }c@{\ }c@{\ }c@{\ }||@{\ }c@{}}
\hline 20&20&20&00&0\\
02&02&00&20&0\\
22&00&02&02&0\\
\end{array} \right).$

Size-$8$ codes of type $\ZZ_4^1\ZZ_2^1$:\\
$ 1:$~$\left( \begin{array}{@{}c@{\ }c@{\ }c@{\ }c@{\ }||@{\ }c@{}}
21&21&21&00&0\\
\hline 20&20&00&20&0\\
\end{array} \right),$\quad
$ 2:$~$\left( \begin{array}{@{}c@{\ }c@{\ }c@{\ }c@{\ }||@{\ }c@{}}
21&20&10&10&0\\
\hline 20&02&02&00&0\\
\end{array} \right),$\quad
$ 3:$~$\left( \begin{array}{@{}c@{\ }c@{\ }c@{\ }c@{\ }||@{\ }c@{}}
21&20&10&10&0\\
\hline 00&02&02&02&0\\
\end{array} \right),$\quad
$ 4:$~$\left( \begin{array}{@{}c@{\ }c@{\ }c@{\ }c@{\ }||@{\ }c@{}}
21&21&10&10&0\\
\hline 02&20&02&00&0\\
\end{array} \right),$\quad
$ 5:$~$\left( \begin{array}{@{}c@{\ }c@{\ }c@{\ }c@{\ }||@{\ }c@{}}
21&21&10&10&0\\
\hline 02&00&02&02&0\\
\end{array} \right),$\quad
$ 6:$~$\left( \begin{array}{@{}c@{\ }c@{\ }c@{\ }c@{\ }||@{\ }c@{}}
21&21&10&10&0\\
\hline 20&20&02&02&0\\
\end{array} \right),$\quad
$ 7:$~$\left( \begin{array}{@{}c@{\ }c@{\ }c@{\ }c@{\ }||@{\ }c@{}}
21&21&10&00&2\\
\hline 20&20&00&20&0\\
\end{array} \right),$\quad
$ 8:$~$\left( \begin{array}{@{}c@{\ }c@{\ }c@{\ }c@{\ }||@{\ }c@{}}
21&21&10&00&2\\
\hline 20&00&02&20&0\\
\end{array} \right),$\quad
$ 9:$~$\left( \begin{array}{@{}c@{\ }c@{\ }c@{\ }c@{\ }||@{\ }c@{}}
20&10&10&10&2\\
\hline 02&02&02&00&0\\
\end{array} \right),$\quad
$ 10:$~$\left( \begin{array}{@{}c@{\ }c@{\ }c@{\ }c@{\ }||@{\ }c@{}}
21&10&10&10&2\\
\hline 20&02&20&00&0\\
\end{array} \right),$\quad
$ 11:$~$\left( \begin{array}{@{}c@{\ }c@{\ }c@{\ }c@{\ }||@{\ }c@{}}
21&10&10&10&2\\
\hline 02&02&02&00&0\\
\end{array} \right),$\quad
$ 12:$~$\left( \begin{array}{@{}c@{\ }c@{\ }c@{\ }c@{\ }||@{\ }c@{}}
21&10&10&10&2\\
\hline 20&02&02&02&0\\
\end{array} \right),$\quad
$ 13:$~$\left( \begin{array}{@{}c@{\ }c@{\ }c@{\ }c@{\ }||@{\ }c@{}}
21&21&20&10&2\\
\hline 20&00&02&02&0\\
\end{array} \right),$\quad
$ 14:$~$\left( \begin{array}{@{}c@{\ }c@{\ }c@{\ }c@{\ }||@{\ }c@{}}
21&21&21&10&2\\
\hline 02&20&00&02&0\\
\end{array} \right).$

Size-$16$ codes of type $\ZZ_4^0\ZZ_2^4$:\\
$ 1:$~$\left( \begin{array}{@{}c@{\ }c@{\ }c@{\ }c@{\ }||@{\ }c@{}}
\hline 20&20&20&00&0\\
02&02&02&00&0\\
02&20&00&20&0\\
22&02&00&02&0\\
\end{array} \right).$

Size-$16$ codes of type $\ZZ_4^1\ZZ_2^2$:\\
$ 1:$~$\left( \begin{array}{@{}c@{\ }c@{\ }c@{\ }c@{\ }||@{\ }c@{}}
21&21&21&00&0\\
\hline 20&20&00&20&0\\
20&00&20&02&0\\
\end{array} \right),$\quad
$ 2:$~$\left( \begin{array}{@{}c@{\ }c@{\ }c@{\ }c@{\ }||@{\ }c@{}}
21&21&21&00&0\\
\hline 20&20&00&20&0\\
22&00&20&02&0\\
\end{array} \right),$\quad
$ 3:$~$\left( \begin{array}{@{}c@{\ }c@{\ }c@{\ }c@{\ }||@{\ }c@{}}
21&21&10&10&0\\
\hline 02&20&02&00&0\\
20&02&00&02&0\\
\end{array} \right),$\quad
$ 4:$~$\left( \begin{array}{@{}c@{\ }c@{\ }c@{\ }c@{\ }||@{\ }c@{}}
21&21&10&10&0\\
\hline 02&20&02&00&0\\
02&22&00&02&0\\
\end{array} \right),$\quad
$ 5:$~$\left( \begin{array}{@{}c@{\ }c@{\ }c@{\ }c@{\ }||@{\ }c@{}}
21&21&10&00&2\\
\hline 20&20&00&20&0\\
20&00&02&02&0\\
\end{array} \right),$\quad
$ 6:$~$\left( \begin{array}{@{}c@{\ }c@{\ }c@{\ }c@{\ }||@{\ }c@{}}
21&21&10&00&2\\
\hline 20&20&00&20&0\\
22&00&02&02&0\\
\end{array} \right),$\quad
$ 7:$~$\left( \begin{array}{@{}c@{\ }c@{\ }c@{\ }c@{\ }||@{\ }c@{}}
21&10&10&10&2\\
\hline 20&02&20&00&0\\
22&20&02&00&0\\
\end{array} \right),$\quad
$ 8:$~$\left( \begin{array}{@{}c@{\ }c@{\ }c@{\ }c@{\ }||@{\ }c@{}}
21&10&10&10&2\\
\hline 20&02&20&00&0\\
02&00&02&02&0\\
\end{array} \right),$\quad
$ 9:$~$\left( \begin{array}{@{}c@{\ }c@{\ }c@{\ }c@{\ }||@{\ }c@{}}
21&10&10&10&2\\
\hline 02&02&02&00&0\\
02&22&00&02&0\\
\end{array} \right).$

Size-$16$ codes of type $\ZZ_4^2\ZZ_2^0$:\\
$ 1:$~$\left( \begin{array}{@{}c@{\ }c@{\ }c@{\ }c@{\ }||@{\ }c@{}}
21&21&21&00&0\\
31&31&00&21&0\\
\hline \end{array} \right),$\quad
$ 2:$~$\left( \begin{array}{@{}c@{\ }c@{\ }c@{\ }c@{\ }||@{\ }c@{}}
21&21&21&00&0\\
30&20&10&21&0\\
\hline \end{array} \right),$\quad
$ 3:$~$\left( \begin{array}{@{}c@{\ }c@{\ }c@{\ }c@{\ }||@{\ }c@{}}
21&21&21&00&0\\
31&31&00&10&2\\
\hline \end{array} \right),$\quad
$ 4:$~$\left( \begin{array}{@{}c@{\ }c@{\ }c@{\ }c@{\ }||@{\ }c@{}}
21&21&21&00&0\\
30&20&10&10&2\\
\hline \end{array} \right),$\quad
$ 5:$~$\left( \begin{array}{@{}c@{\ }c@{\ }c@{\ }c@{\ }||@{\ }c@{}}
21&21&21&00&0\\
12&10&00&21&2\\
\hline \end{array} \right),$\quad
$ 6:$~$\left( \begin{array}{@{}c@{\ }c@{\ }c@{\ }c@{\ }||@{\ }c@{}}
21&21&21&00&0\\
31&20&10&21&2\\
\hline \end{array} \right),$\quad
$ 7:$~$\left( \begin{array}{@{}c@{\ }c@{\ }c@{\ }c@{\ }||@{\ }c@{}}
21&20&10&10&0\\
20&01&21&01&0\\
\hline \end{array} \right),$\quad
$ 8:$~$\left( \begin{array}{@{}c@{\ }c@{\ }c@{\ }c@{\ }||@{\ }c@{}}
21&20&10&10&0\\
12&01&02&01&0\\
\hline \end{array} \right),$\quad
$ 9:$~$\left( \begin{array}{@{}c@{\ }c@{\ }c@{\ }c@{\ }||@{\ }c@{}}
21&20&10&10&0\\
31&21&01&00&2\\
\hline \end{array} \right),$\quad
$ 10:$~$\left( \begin{array}{@{}c@{\ }c@{\ }c@{\ }c@{\ }||@{\ }c@{}}
21&20&10&10&0\\
31&01&21&00&2\\
\hline \end{array} \right),$\quad
$ 11:$~$\left( \begin{array}{@{}c@{\ }c@{\ }c@{\ }c@{\ }||@{\ }c@{}}
21&20&10&10&0\\
20&01&01&01&2\\
\hline \end{array} \right),$\quad
$ 12:$~$\left( \begin{array}{@{}c@{\ }c@{\ }c@{\ }c@{\ }||@{\ }c@{}}
21&20&10&10&0\\
21&01&11&01&2\\
\hline \end{array} \right),$\quad
$ 13:$~$\left( \begin{array}{@{}c@{\ }c@{\ }c@{\ }c@{\ }||@{\ }c@{}}
21&20&10&10&0\\
02&01&11&01&2\\
\hline \end{array} \right),$\quad
$ 14:$~$\left( \begin{array}{@{}c@{\ }c@{\ }c@{\ }c@{\ }||@{\ }c@{}}
21&20&10&10&0\\
01&01&21&01&2\\
\hline \end{array} \right),$\quad
$ 15:$~$\left( \begin{array}{@{}c@{\ }c@{\ }c@{\ }c@{\ }||@{\ }c@{}}
21&20&10&10&0\\
10&01&02&01&2\\
\hline \end{array} \right),$\quad
$ 16:$~$\left( \begin{array}{@{}c@{\ }c@{\ }c@{\ }c@{\ }||@{\ }c@{}}
21&20&10&10&0\\
12&21&02&01&2\\
\hline \end{array} \right),$\quad
$ 17:$~$\left( \begin{array}{@{}c@{\ }c@{\ }c@{\ }c@{\ }||@{\ }c@{}}
21&21&10&10&0\\
31&31&01&01&0\\
\hline \end{array} \right),$\quad
$ 18:$~$\left( \begin{array}{@{}c@{\ }c@{\ }c@{\ }c@{\ }||@{\ }c@{}}
21&21&10&10&0\\
12&31&11&01&0\\
\hline \end{array} \right),$\quad
$ 19:$~$\left( \begin{array}{@{}c@{\ }c@{\ }c@{\ }c@{\ }||@{\ }c@{}}
21&21&10&10&0\\
12&10&21&01&0\\
\hline \end{array} \right),$\quad
$ 20:$~$\left( \begin{array}{@{}c@{\ }c@{\ }c@{\ }c@{\ }||@{\ }c@{}}
21&21&10&10&0\\
12&21&01&00&2\\
\hline \end{array} \right),$\quad
$ 21:$~$\left( \begin{array}{@{}c@{\ }c@{\ }c@{\ }c@{\ }||@{\ }c@{}}
21&21&10&10&0\\
31&01&21&00&2\\
\hline \end{array} \right),$\quad
$ 22:$~$\left( \begin{array}{@{}c@{\ }c@{\ }c@{\ }c@{\ }||@{\ }c@{}}
21&21&10&10&0\\
12&10&01&01&2\\
\hline \end{array} \right),$\quad
$ 23:$~$\left( \begin{array}{@{}c@{\ }c@{\ }c@{\ }c@{\ }||@{\ }c@{}}
21&21&10&10&0\\
01&20&01&01&2\\
\hline \end{array} \right),$\quad
$ 24:$~$\left( \begin{array}{@{}c@{\ }c@{\ }c@{\ }c@{\ }||@{\ }c@{}}
21&21&10&10&0\\
31&30&01&01&2\\
\hline \end{array} \right),$\quad
$ 25:$~$\left( \begin{array}{@{}c@{\ }c@{\ }c@{\ }c@{\ }||@{\ }c@{}}
21&21&10&10&0\\
31&10&11&01&2\\
\hline \end{array} \right),$\quad
$ 26:$~$\left( \begin{array}{@{}c@{\ }c@{\ }c@{\ }c@{\ }||@{\ }c@{}}
21&21&10&10&0\\
02&01&11&01&2\\
\hline \end{array} \right),$\quad
$ 27:$~$\left( \begin{array}{@{}c@{\ }c@{\ }c@{\ }c@{\ }||@{\ }c@{}}
21&21&10&10&0\\
21&00&21&01&2\\
\hline \end{array} \right),$\quad
$ 28:$~$\left( \begin{array}{@{}c@{\ }c@{\ }c@{\ }c@{\ }||@{\ }c@{}}
21&21&10&10&0\\
23&00&21&01&2\\
\hline \end{array} \right),$\quad
$ 29:$~$\left( \begin{array}{@{}c@{\ }c@{\ }c@{\ }c@{\ }||@{\ }c@{}}
21&21&10&10&0\\
01&10&02&01&2\\
\hline \end{array} \right),$\quad
$ 30:$~$\left( \begin{array}{@{}c@{\ }c@{\ }c@{\ }c@{\ }||@{\ }c@{}}
21&21&10&10&0\\
31&21&02&01&2\\
\hline \end{array} \right).$

Size-$32$ codes of type $\ZZ_4^0\ZZ_2^5$:
none.

Size-$32$ codes of type $\ZZ_4^1\ZZ_2^3$:\\
$ 1:$~$\left( \begin{array}{@{}c@{\ }c@{\ }c@{\ }c@{\ }||@{\ }c@{}}
21&10&10&10&2\\
\hline 20&02&20&00&0\\
22&20&02&00&0\\
20&20&00&02&0\\
\end{array} \right).$

Size-$32$ codes of type $\ZZ_4^2\ZZ_2^1$:\\
$ 1:$~$\left( \begin{array}{@{}c@{\ }c@{\ }c@{\ }c@{\ }||@{\ }c@{}}
21&21&21&00&0\\
31&31&00&10&2\\
\hline 20&00&20&02&0\\
\end{array} \right),$\quad
$ 2:$~$\left( \begin{array}{@{}c@{\ }c@{\ }c@{\ }c@{\ }||@{\ }c@{}}
21&21&21&00&0\\
30&20&10&10&2\\
\hline 22&20&00&02&0\\
\end{array} \right),$\quad
$ 3:$~$\left( \begin{array}{@{}c@{\ }c@{\ }c@{\ }c@{\ }||@{\ }c@{}}
21&21&10&10&0\\
12&31&11&01&0\\
\hline 20&02&02&00&0\\
\end{array} \right),$\quad
$ 4:$~$\left( \begin{array}{@{}c@{\ }c@{\ }c@{\ }c@{\ }||@{\ }c@{}}
21&21&10&10&0\\
12&21&01&00&2\\
\hline 02&20&00&02&0\\
\end{array} \right),$\quad
$ 5:$~$\left( \begin{array}{@{}c@{\ }c@{\ }c@{\ }c@{\ }||@{\ }c@{}}
21&21&10&10&0\\
12&10&01&01&2\\
\hline 22&02&02&00&0\\
\end{array} \right).$

Size-$64$ codes:
none.
\end{small}

\subsection{\texorpdfstring{In $D(3,2+1)$}{In D(3,2+1)}}
\begin{small}\raggedright

Size-$2$ codes of type $\ZZ_4^0\ZZ_2^1$:\\
$ 1:$~$\left( \begin{array}{@{}c@{\ }c@{\ }c@{\ }|@{\ }c@{\ }c@{\ }|@{\ }c@{}}
\hline 20&20&00&10&10&0\\
\end{array} \right),$\quad
$ 2:$~$\left( \begin{array}{@{}c@{\ }c@{\ }c@{\ }|@{\ }c@{\ }c@{\ }|@{\ }c@{}}
\hline 20&20&20&00&00&0\\
\end{array} \right),$\quad
$ 3:$~$\left( \begin{array}{@{}c@{\ }c@{\ }c@{\ }|@{\ }c@{\ }c@{\ }|@{\ }c@{}}
\hline 20&20&20&10&10&0\\
\end{array} \right),$\quad
$ 4:$~$\left( \begin{array}{@{}c@{\ }c@{\ }c@{\ }|@{\ }c@{\ }c@{\ }|@{\ }c@{}}
\hline 20&20&00&10&00&2\\
\end{array} \right),$\quad
$ 5:$~$\left( \begin{array}{@{}c@{\ }c@{\ }c@{\ }|@{\ }c@{\ }c@{\ }|@{\ }c@{}}
\hline 20&20&20&10&00&2\\
\end{array} \right).$

Size-$4$ codes of type $\ZZ_4^0\ZZ_2^2$:\\
$ 1:$~$\left( \begin{array}{@{}c@{\ }c@{\ }c@{\ }|@{\ }c@{\ }c@{\ }|@{\ }c@{}}
\hline 20&20&00&10&10&0\\
02&02&00&01&01&0\\
\end{array} \right),$\quad
$ 2:$~$\left( \begin{array}{@{}c@{\ }c@{\ }c@{\ }|@{\ }c@{\ }c@{\ }|@{\ }c@{}}
\hline 20&20&00&10&10&0\\
20&00&20&01&01&0\\
\end{array} \right),$\quad
$ 3:$~$\left( \begin{array}{@{}c@{\ }c@{\ }c@{\ }|@{\ }c@{\ }c@{\ }|@{\ }c@{}}
\hline 20&20&00&10&10&0\\
02&00&20&10&10&0\\
\end{array} \right),$\quad
$ 4:$~$\left( \begin{array}{@{}c@{\ }c@{\ }c@{\ }|@{\ }c@{\ }c@{\ }|@{\ }c@{}}
\hline 20&20&00&10&10&0\\
02&00&20&01&01&0\\
\end{array} \right),$\quad
$ 5:$~$\left( \begin{array}{@{}c@{\ }c@{\ }c@{\ }|@{\ }c@{\ }c@{\ }|@{\ }c@{}}
\hline 20&20&00&10&10&0\\
02&02&20&00&00&0\\
\end{array} \right),$\quad
$ 6:$~$\left( \begin{array}{@{}c@{\ }c@{\ }c@{\ }|@{\ }c@{\ }c@{\ }|@{\ }c@{}}
\hline 20&20&00&10&10&0\\
02&02&20&01&01&0\\
\end{array} \right),$\quad
$ 7:$~$\left( \begin{array}{@{}c@{\ }c@{\ }c@{\ }|@{\ }c@{\ }c@{\ }|@{\ }c@{}}
\hline 20&20&00&10&10&0\\
02&02&00&10&00&2\\
\end{array} \right),$\quad
$ 8:$~$\left( \begin{array}{@{}c@{\ }c@{\ }c@{\ }|@{\ }c@{\ }c@{\ }|@{\ }c@{}}
\hline 20&20&00&10&10&0\\
20&00&20&10&00&2\\
\end{array} \right),$\quad
$ 9:$~$\left( \begin{array}{@{}c@{\ }c@{\ }c@{\ }|@{\ }c@{\ }c@{\ }|@{\ }c@{}}
\hline 20&20&00&10&10&0\\
02&00&20&10&00&2\\
\end{array} \right),$\quad
$ 10:$~$\left( \begin{array}{@{}c@{\ }c@{\ }c@{\ }|@{\ }c@{\ }c@{\ }|@{\ }c@{}}
\hline 20&20&00&10&10&0\\
02&02&20&10&00&2\\
\end{array} \right),$\quad
$ 11:$~$\left( \begin{array}{@{}c@{\ }c@{\ }c@{\ }|@{\ }c@{\ }c@{\ }|@{\ }c@{}}
\hline 20&20&20&00&00&0\\
02&02&02&00&00&0\\
\end{array} \right),$\quad
$ 12:$~$\left( \begin{array}{@{}c@{\ }c@{\ }c@{\ }|@{\ }c@{\ }c@{\ }|@{\ }c@{}}
\hline 20&20&20&00&00&0\\
02&02&02&10&10&0\\
\end{array} \right),$\quad
$ 13:$~$\left( \begin{array}{@{}c@{\ }c@{\ }c@{\ }|@{\ }c@{\ }c@{\ }|@{\ }c@{}}
\hline 20&20&20&00&00&0\\
02&20&00&10&00&2\\
\end{array} \right),$\quad
$ 14:$~$\left( \begin{array}{@{}c@{\ }c@{\ }c@{\ }|@{\ }c@{\ }c@{\ }|@{\ }c@{}}
\hline 20&20&20&00&00&0\\
02&02&00&10&00&2\\
\end{array} \right),$\quad
$ 15:$~$\left( \begin{array}{@{}c@{\ }c@{\ }c@{\ }|@{\ }c@{\ }c@{\ }|@{\ }c@{}}
\hline 20&20&20&00&00&0\\
02&02&02&10&00&2\\
\end{array} \right),$\quad
$ 16:$~$\left( \begin{array}{@{}c@{\ }c@{\ }c@{\ }|@{\ }c@{\ }c@{\ }|@{\ }c@{}}
\hline 20&20&20&10&10&0\\
02&02&02&01&01&0\\
\end{array} \right),$\quad
$ 17:$~$\left( \begin{array}{@{}c@{\ }c@{\ }c@{\ }|@{\ }c@{\ }c@{\ }|@{\ }c@{}}
\hline 20&20&20&10&10&0\\
02&20&00&10&00&2\\
\end{array} \right),$\quad
$ 18:$~$\left( \begin{array}{@{}c@{\ }c@{\ }c@{\ }|@{\ }c@{\ }c@{\ }|@{\ }c@{}}
\hline 20&20&20&10&10&0\\
02&02&00&10&00&2\\
\end{array} \right),$\quad
$ 19:$~$\left( \begin{array}{@{}c@{\ }c@{\ }c@{\ }|@{\ }c@{\ }c@{\ }|@{\ }c@{}}
\hline 20&20&20&10&10&0\\
02&02&02&10&00&2\\
\end{array} \right).$

Size-$4$ codes of type $\ZZ_4^1\ZZ_2^0$:\\
$ 1:$~$\left( \begin{array}{@{}c@{\ }c@{\ }c@{\ }|@{\ }c@{\ }c@{\ }|@{\ }c@{}}
21&10&10&10&10&0\\
\hline \end{array} \right),$\quad
$ 2:$~$\left( \begin{array}{@{}c@{\ }c@{\ }c@{\ }|@{\ }c@{\ }c@{\ }|@{\ }c@{}}
21&21&10&10&00&0\\
\hline \end{array} \right),$\quad
$ 3:$~$\left( \begin{array}{@{}c@{\ }c@{\ }c@{\ }|@{\ }c@{\ }c@{\ }|@{\ }c@{}}
21&21&21&00&00&0\\
\hline \end{array} \right),$\quad
$ 4:$~$\left( \begin{array}{@{}c@{\ }c@{\ }c@{\ }|@{\ }c@{\ }c@{\ }|@{\ }c@{}}
21&21&21&10&10&0\\
\hline \end{array} \right),$\quad
$ 5:$~$\left( \begin{array}{@{}c@{\ }c@{\ }c@{\ }|@{\ }c@{\ }c@{\ }|@{\ }c@{}}
10&10&10&10&10&2\\
\hline \end{array} \right),$\quad
$ 6:$~$\left( \begin{array}{@{}c@{\ }c@{\ }c@{\ }|@{\ }c@{\ }c@{\ }|@{\ }c@{}}
21&10&10&10&00&2\\
\hline \end{array} \right),$\quad
$ 7:$~$\left( \begin{array}{@{}c@{\ }c@{\ }c@{\ }|@{\ }c@{\ }c@{\ }|@{\ }c@{}}
21&21&10&00&00&2\\
\hline \end{array} \right),$\quad
$ 8:$~$\left( \begin{array}{@{}c@{\ }c@{\ }c@{\ }|@{\ }c@{\ }c@{\ }|@{\ }c@{}}
21&21&10&10&10&2\\
\hline \end{array} \right),$\quad
$ 9:$~$\left( \begin{array}{@{}c@{\ }c@{\ }c@{\ }|@{\ }c@{\ }c@{\ }|@{\ }c@{}}
21&21&21&10&00&2\\
\hline \end{array} \right).$

Size-$8$ codes of type $\ZZ_4^0\ZZ_2^3$:\\
$ 1:$~$\left( \begin{array}{@{}c@{\ }c@{\ }c@{\ }|@{\ }c@{\ }c@{\ }|@{\ }c@{}}
\hline 20&20&00&10&10&0\\
02&02&00&01&01&0\\
20&00&20&01&01&0\\
\end{array} \right),$\quad
$ 2:$~$\left( \begin{array}{@{}c@{\ }c@{\ }c@{\ }|@{\ }c@{\ }c@{\ }|@{\ }c@{}}
\hline 20&20&00&10&10&0\\
20&00&20&01&01&0\\
02&02&02&00&00&0\\
\end{array} \right),$\quad
$ 3:$~$\left( \begin{array}{@{}c@{\ }c@{\ }c@{\ }|@{\ }c@{\ }c@{\ }|@{\ }c@{}}
\hline 20&20&00&10&10&0\\
02&00&20&10&10&0\\
00&02&02&10&10&0\\
\end{array} \right),$\quad
$ 4:$~$\left( \begin{array}{@{}c@{\ }c@{\ }c@{\ }|@{\ }c@{\ }c@{\ }|@{\ }c@{}}
\hline 20&20&00&10&10&0\\
02&00&20&10&10&0\\
00&02&02&01&01&0\\
\end{array} \right),$\quad
$ 5:$~$\left( \begin{array}{@{}c@{\ }c@{\ }c@{\ }|@{\ }c@{\ }c@{\ }|@{\ }c@{}}
\hline 20&20&00&10&10&0\\
02&00&20&10&10&0\\
02&02&00&10&00&2\\
\end{array} \right),$\quad
$ 6:$~$\left( \begin{array}{@{}c@{\ }c@{\ }c@{\ }|@{\ }c@{\ }c@{\ }|@{\ }c@{}}
\hline 20&20&00&10&10&0\\
02&00&20&10&10&0\\
00&02&02&10&00&2\\
\end{array} \right),$\quad
$ 7:$~$\left( \begin{array}{@{}c@{\ }c@{\ }c@{\ }|@{\ }c@{\ }c@{\ }|@{\ }c@{}}
\hline 20&20&00&10&10&0\\
02&00&20&01&01&0\\
00&02&02&11&11&0\\
\end{array} \right),$\quad
$ 8:$~$\left( \begin{array}{@{}c@{\ }c@{\ }c@{\ }|@{\ }c@{\ }c@{\ }|@{\ }c@{}}
\hline 20&20&00&10&10&0\\
02&02&20&00&00&0\\
22&02&00&10&00&2\\
\end{array} \right),$\quad
$ 9:$~$\left( \begin{array}{@{}c@{\ }c@{\ }c@{\ }|@{\ }c@{\ }c@{\ }|@{\ }c@{}}
\hline 20&20&00&10&10&0\\
02&02&20&00&00&0\\
20&00&02&10&00&2\\
\end{array} \right),$\quad
$ 10:$~$\left( \begin{array}{@{}c@{\ }c@{\ }c@{\ }|@{\ }c@{\ }c@{\ }|@{\ }c@{}}
\hline 20&20&00&10&10&0\\
02&02&20&00&00&0\\
02&00&02&10&00&2\\
\end{array} \right),$\quad
$ 11:$~$\left( \begin{array}{@{}c@{\ }c@{\ }c@{\ }|@{\ }c@{\ }c@{\ }|@{\ }c@{}}
\hline 20&20&00&10&10&0\\
02&02&20&00&00&0\\
22&00&02&10&00&2\\
\end{array} \right),$\quad
$ 12:$~$\left( \begin{array}{@{}c@{\ }c@{\ }c@{\ }|@{\ }c@{\ }c@{\ }|@{\ }c@{}}
\hline 20&20&20&00&00&0\\
02&02&02&00&00&0\\
02&20&00&10&00&2\\
\end{array} \right),$\quad
$ 13:$~$\left( \begin{array}{@{}c@{\ }c@{\ }c@{\ }|@{\ }c@{\ }c@{\ }|@{\ }c@{}}
\hline 20&20&20&00&00&0\\
02&02&02&10&10&0\\
02&20&00&10&00&2\\
\end{array} \right).$

Size-$8$ codes of type $\ZZ_4^1\ZZ_2^1$:\\
$ 1:$~$\left( \begin{array}{@{}c@{\ }c@{\ }c@{\ }|@{\ }c@{\ }c@{\ }|@{\ }c@{}}
21&10&10&10&10&0\\
\hline 20&02&00&01&01&0\\
\end{array} \right),$\quad
$ 2:$~$\left( \begin{array}{@{}c@{\ }c@{\ }c@{\ }|@{\ }c@{\ }c@{\ }|@{\ }c@{}}
21&10&10&10&10&0\\
\hline 02&02&00&01&10&0\\
\end{array} \right),$\quad
$ 3:$~$\left( \begin{array}{@{}c@{\ }c@{\ }c@{\ }|@{\ }c@{\ }c@{\ }|@{\ }c@{}}
21&10&10&10&10&0\\
\hline 00&02&02&10&10&0\\
\end{array} \right),$\quad
$ 4:$~$\left( \begin{array}{@{}c@{\ }c@{\ }c@{\ }|@{\ }c@{\ }c@{\ }|@{\ }c@{}}
21&10&10&10&10&0\\
\hline 00&02&02&01&01&0\\
\end{array} \right),$\quad
$ 5:$~$\left( \begin{array}{@{}c@{\ }c@{\ }c@{\ }|@{\ }c@{\ }c@{\ }|@{\ }c@{}}
21&10&10&10&10&0\\
\hline 20&02&02&01&10&0\\
\end{array} \right),$\quad
$ 6:$~$\left( \begin{array}{@{}c@{\ }c@{\ }c@{\ }|@{\ }c@{\ }c@{\ }|@{\ }c@{}}
21&10&10&10&10&0\\
\hline 20&20&00&01&00&2\\
\end{array} \right),$\quad
$ 7:$~$\left( \begin{array}{@{}c@{\ }c@{\ }c@{\ }|@{\ }c@{\ }c@{\ }|@{\ }c@{}}
21&10&10&10&10&0\\
\hline 20&02&00&10&00&2\\
\end{array} \right),$\quad
$ 8:$~$\left( \begin{array}{@{}c@{\ }c@{\ }c@{\ }|@{\ }c@{\ }c@{\ }|@{\ }c@{}}
21&10&10&10&10&0\\
\hline 02&02&00&01&00&2\\
\end{array} \right),$\quad
$ 9:$~$\left( \begin{array}{@{}c@{\ }c@{\ }c@{\ }|@{\ }c@{\ }c@{\ }|@{\ }c@{}}
21&10&10&10&10&0\\
\hline 00&02&02&10&00&2\\
\end{array} \right),$\quad
$ 10:$~$\left( \begin{array}{@{}c@{\ }c@{\ }c@{\ }|@{\ }c@{\ }c@{\ }|@{\ }c@{}}
21&10&10&10&10&0\\
\hline 20&02&02&01&00&2\\
\end{array} \right),$\quad
$ 11:$~$\left( \begin{array}{@{}c@{\ }c@{\ }c@{\ }|@{\ }c@{\ }c@{\ }|@{\ }c@{}}
21&21&10&10&00&0\\
\hline 20&00&02&10&10&0\\
\end{array} \right),$\quad
$ 12:$~$\left( \begin{array}{@{}c@{\ }c@{\ }c@{\ }|@{\ }c@{\ }c@{\ }|@{\ }c@{}}
21&21&10&10&00&0\\
\hline 02&00&02&01&10&0\\
\end{array} \right),$\quad
$ 13:$~$\left( \begin{array}{@{}c@{\ }c@{\ }c@{\ }|@{\ }c@{\ }c@{\ }|@{\ }c@{}}
21&21&10&10&00&0\\
\hline 20&20&00&00&10&2\\
\end{array} \right),$\quad
$ 14:$~$\left( \begin{array}{@{}c@{\ }c@{\ }c@{\ }|@{\ }c@{\ }c@{\ }|@{\ }c@{}}
21&21&10&10&00&0\\
\hline 02&20&00&01&00&2\\
\end{array} \right),$\quad
$ 15:$~$\left( \begin{array}{@{}c@{\ }c@{\ }c@{\ }|@{\ }c@{\ }c@{\ }|@{\ }c@{}}
21&21&10&10&00&0\\
\hline 20&00&02&10&00&2\\
\end{array} \right),$\quad
$ 16:$~$\left( \begin{array}{@{}c@{\ }c@{\ }c@{\ }|@{\ }c@{\ }c@{\ }|@{\ }c@{}}
21&21&10&10&00&0\\
\hline 20&00&02&00&10&2\\
\end{array} \right),$\quad
$ 17:$~$\left( \begin{array}{@{}c@{\ }c@{\ }c@{\ }|@{\ }c@{\ }c@{\ }|@{\ }c@{}}
21&21&10&10&00&0\\
\hline 02&00&02&01&00&2\\
\end{array} \right),$\quad
$ 18:$~$\left( \begin{array}{@{}c@{\ }c@{\ }c@{\ }|@{\ }c@{\ }c@{\ }|@{\ }c@{}}
21&21&10&10&00&0\\
\hline 20&20&02&01&00&2\\
\end{array} \right),$\quad
$ 19:$~$\left( \begin{array}{@{}c@{\ }c@{\ }c@{\ }|@{\ }c@{\ }c@{\ }|@{\ }c@{}}
21&21&21&00&00&0\\
\hline 20&20&00&10&00&2\\
\end{array}%
 \right).$

Size-$16$ codes of type $\ZZ_4^2\ZZ_2^0$:
none.

Size-$32$ codes:
none.
\end{small}

\subsection{\texorpdfstring{In $D(3,1+2)$}{In D(3,1+2)}}
\begin{small}\raggedright

Size-$2$ codes of type $\ZZ_4^0\ZZ_2^1$:\\
$ 1:$~$\left( %
 \right).$

Size-$32$ codes of type $\ZZ_4^0\ZZ_2^5$:
none.

Size-$32$ codes of type $\ZZ_4^1\ZZ_2^3$:\\
$ 1:$~$\left( %
 \right).$

Size-$64$ codes:
none.
\end{small}

\subsection{\texorpdfstring{In $D(2,3+2)$}{In D(2,3+2)}}
\begin{small}\raggedright

Size-$2$ codes of type $\ZZ_4^0\ZZ_2^1$:\\
$ 1:$~$\left( %
 \right).$

Size-$16$ codes of type $\ZZ_4^2\ZZ_2^0$:
none.

Size-$32$ codes of type $\ZZ_4^0\ZZ_2^5$:
none.

Size-$32$ codes of type $\ZZ_4^1\ZZ_2^3$:\\
$ 1:$~$\left( %
 \right).$

Size-$32$ codes:
none.
\end{small}

\subsection{\texorpdfstring{In $D(3,0+3)$}{In D(3,0+3)}}
\begin{small}\raggedright

Size-$2$ codes of type $\ZZ_4^0\ZZ_2^1$:\\
$ 1:$~$\left( %
 \right).$

Size-$64$ codes of type $\ZZ_4^0\ZZ_2^6$:
none.

Size-$64$ codes of type $\ZZ_4^1\ZZ_2^4$:
none.

Size-$64$ codes of type $\ZZ_4^2\ZZ_2^2$:\\
$ 1:$~$\left( %
 \right).$
\end{small}

\subsection{\texorpdfstring{In $D(2,2+3)$}{In D(2,2+3)}}
\begin{small}\raggedright

Size-$2$ codes of type $\ZZ_4^0\ZZ_2^1$:\\
$ 1:$~$\left( %
 \right).$

Size-$16$ codes of type $\ZZ_4^2\ZZ_2^0$:
none.

Size-$32$ codes:
none.
\end{small}

\subsection{\texorpdfstring{In $D(2,1+4)$}{In D(2,1+4)}}
\begin{small}\raggedright

Size-$2$ codes of type $\ZZ_4^0\ZZ_2^1$:\\
$ 1:$~$\left( %
 \right).$

Size-$64$ codes:
none.
\end{small}

\subsection{\texorpdfstring{In $D(1,3+4)$}{In D(1,3+4)}}
\begin{small}\raggedright

Size-$2$ codes of type $\ZZ_4^0\ZZ_2^1$:\\
$ 1:$~$\left( %
 \right).$

Size-$16$ codes of type $\ZZ_4^2\ZZ_2^0$:
none.

Size-$32$ codes:
none.
\end{small}

\subsection{\texorpdfstring{In $D(2,0+5)$}{In D(2,0+5)}}
\begin{small}\raggedright

Size-$2$ codes of type $\ZZ_4^0\ZZ_2^1$:\\
$ 1:$~$\left( %
 \right).$

Size-$32$ codes:
none.
\end{small}

\subsection{\texorpdfstring{In $D(1,2+5)$}{In D(1,2+5)}}
\begin{small}\raggedright

Size-$2$ codes of type $\ZZ_4^0\ZZ_2^1$:\\
$ 1:$~$\left( %
 \right).$

Size-$16$ codes of type $\ZZ_4^2\ZZ_2^0$:
none.

Size-$32$:
none.
\end{small}

\subsection{\texorpdfstring{In $D(1,1+6)$}{In D(1,1+6)}}
\begin{small}\raggedright

Size-$2$ codes of type $\ZZ_4^0\ZZ_2^1$:\\
$ 1:$~$\left( %
 \right).$

Size-$16$ codes:
none.

\end{small}

\subsection{\texorpdfstring{In $D(0,3+6)$}{In D(0,3+6)}}
\begin{small}\raggedright

Size-$2$ codes of type $\ZZ_4^0\ZZ_2^1$:\\
$ 1:$~$\left( %
 \right).$

Size-$16$ codes:
none.
\end{small}

\subsection{\texorpdfstring{In $D(1,0+7)$}{In D(1,0+7)}}
\begin{small}\raggedright

Size-$2$ codes of type $\ZZ_4^0\ZZ_2^1$:\\
$ 1:$~$\left( %
 \right).$

Size-$32$ codes:
none.
\end{small}

\subsection{\texorpdfstring{In $D(0,2+7)$}{In D(0,2+7)}}
\begin{small}\raggedright

Size-$2$ codes of type $\ZZ_4^0\ZZ_2^1$:\\
$ 1:$~$\left( %
 \right).$

Size-$16$ codes:
none.
\end{small}

\subsection{\texorpdfstring{In $D(0,1+8)$}{In D(0,1+8)}}
\begin{small}\raggedright

Size-$2$ codes of type $\ZZ_4^0\ZZ_2^1$:\\
$ 1:$~$\left( %
 \right).$

Size-$16$ codes:
none.
\end{small}
\subsection{\texorpdfstring{In $D(0,0+9)$}{In D(0,0+9)}}
\begin{small}\raggedright

Size-$2$ codes of type $\ZZ_4^0\ZZ_2^1$:\\
$ 1:$~$\left( %
 \right).$

Size-$8$ codes:
none.
\end{small}

\subsection{\texorpdfstring{In $D(0,9+0)$}{In D(0,9+0)}}
\begin{small}\raggedright

Size-$2$ codes of type $\ZZ_2^1$:
\\
$ 1: \left( %
 \right).$

Size-$4$ codes of type $\ZZ_2^2$:
\\
$ 1: \left( %
 \right).$

Size-$8$ codes of type $\ZZ_2^3$:
\\
$ 1: \left( %
 \right).$

Size-$16$ codes of type $\ZZ_2^4$:
\\
$ 1: \left( %
 \right).$

\end{small}

\section{Additive quaternary codes of length 6 and distance at least 3}\label{a:d3}
\subsection{Dimension 6}\label{a:d3k6}\mbox{}\\[1em]
\newcommand\0{0}
\newcommand\1{1}
\newcommand\2{\omega}
\newcommand\3{\upsilon}
Here, we list generator matrices of all $646$ equivalence classes of additive $[6,3,3]_4$ codes.
For short, we encode $00$, $01$, $10$, $11$ by $\0$, $\1$, $\2$, $\3$, respectively,
but one can also treat $\0$, $\1$, $\2$, $\3=\2^2$ as the elements of the field~$\F_4$.

\begin{small}\raggedright\noindent
$\left(%
\right)$.
\end{small}

\subsection{Dimension 7}\label{a:d3k7}
Below, generator matrices of the all $14$ inequivalent additive $[6,3.5,3]_{2^2}$ codes are listed.
Four of them can be lifted to codes of type $\ZZ_4^2\ZZ_2^5$ in $D(6,0)$ with weights from~$0$, $6$, $8$, $10$, $12$, the remaining ten, only to codes of type $\ZZ_4^1\ZZ_2^6$. (For lifting, the underlined items in the first one or two rows should be increased by~$2$, the remaining rows multiplied by~$2$.)

\begin{small}\raggedright\noindent
\def\0{\underline{0}} \def\1{\underline{1}}
 $$ \left(%
\right).$$
\end{small}


\begin{thebibliography}{10}

\bibitem{graph1024}
\url{https://mathworld.wolfram.com/Shi-Krotov-SoleGraph.html}.

\bibitem{BZZ:1974:UPC}
L.~A. Bassalygo, G.~V. Zaitsev, and V.~A. Zinoviev.
\newblock On uniformly packed codes.
\newblock {\em \href{http://link.springer.com/journal/11122}{Probl. Inf.
  Transm.}}, 10(1):6--9, 1974.
\newblock Translated from \em Probl. Peredachi Inf.\em, 10(1): 9-14, 1974.

\bibitem{BesKro:Ch6}
E.~A. Bespalov and D.~S. Krotov.
\newblock Some completely regular codes in {D}oob graphs.
\newblock In M.~Shi and P.~Sol\'e, editors, {\em Completely Regular Codes in
  Distance Regular Graphs}, chapter~6, pages 379--448. CRC Press, 2025.
\newblock \DOI{https://doi.org/10.1201/9781003393931-6}.

\bibitem{QextNewEdition}
I.~Bouyukliev and S.~Bouyuklieva.
\newblock {\texttt{QextNewEdition}}.
\newblock
  \url{http://www.moi.math.bas.bg/moiuser/~data/Software/QextNewEdition}.

\bibitem{BouOst:2005}
I.~Bouyukliev and P.~R.~J. {\"O}sterg{\aa}rd.
\newblock Classification of self-orthogonal codes over {$F_3$} and {$F_4$}.
\newblock 19(2):363--370.
\newblock \DOI{10.1137/S0895480104441085}.

\bibitem{CRSS:quantum}
A.~R. Calderbank, E.~M. Rains, P.~M. Shor, and N.~J.~A. Sloane.
\newblock Quantum error correction via codes over {GF$(4)$}.
\newblock {\em
  \href{http://ieeexplore.ieee.org/xpl/RecentIssue.jsp?punumber=18}{IEEE Trans.
  Inf. Theory}}, 44(4):1369--1387, July 1998.
\newblock \DOI{10.1109/18.681315}.

\bibitem{DanPar:2006}
L.~E. Danielsen and M.~G. Parker.
\newblock On the classification of all self-dual additive codes over {GF}$(4)$
  of length up to $12$.
\newblock 113(7):1351--1367.

\bibitem{Kro:perfect-doob}
D.~S. Krotov.
\newblock Perfect codes in {D}oob graphs.
\newblock {\em \href{http://link.springer.com/journal/10623}{Des. Codes
  Cryptography}}, 80(1):91--102, July 2016.
\newblock \DOI{10.1007/s10623-015-0066-6}.

\bibitem{KroPot:CRC&EP}
D.~S. Krotov and V.~N. Potapov.
\newblock Completely regular codes and equitable partitions.
\newblock In M.~Shi and P.~Sole\', editors, {\em Completely Regular Codes in
  Distance Regular Graphs}, chapter~1, pages 1--84. CRC Press, 2025.
\newblock \url{https://doi.org/10.1201/9781003393931-1}.

\bibitem{MWOSW:78}
F.~J. MacWilliams, A.~M. Odlyzko, N.~J.~A. Sloane, and H.~N. Ward.
\newblock Self-dual codes over {GF}(4).
\newblock 25:288--318.
\newblock \DOI{10.1016/0097-3165(78)90021-3}.

\bibitem{SKS:drg}
M.~Shi, D.~S. Krotov, and P.~Sol\'e.
\newblock A new distance-regular graph of diameter $3$ on $1024$ vertices.
\newblock {\em \href{http://link.springer.com/journal/10623}{Des. Codes
  Cryptography}}, 87(9):2091--2101, 2019.
\newblock \DOI{10.1007/s10623-019-00609-w}.

\bibitem{CRCinDRG}
M.~Shi and P.~Sol\'e, editors.
\newblock {\em Completely Regular Codes in Distance Regular Graphs}.
\newblock Chapman \& Hall/CRC Monographs and Research Notes in Mathematics. CRC
  Press, Boca Raton FL, 2025.

\end{thebibliography}
\end{document}